\documentclass[final,3p,times,sort&compress]{elsarticle}
\usepackage{lineno,hyperref}
\modulolinenumbers[5]
\usepackage{caption}
\usepackage{subcaption}
\usepackage{framed}
\usepackage{amsmath}
\usepackage{bm}
\usepackage{dsfont}
\usepackage{graphicx,color,url}
\usepackage{epstopdf}
\usepackage{multirow}
\usepackage{chngpage}
\usepackage{array}
\usepackage{float}
\usepackage{lineno}
\usepackage{nomencl} 
\makenomenclature
\usepackage{float}
\usepackage{setspace} 
\usepackage[english]{babel}
\usepackage{amsthm}
\newtheorem{theorem}{Theorem}
\newdefinition{definition}{Definition}
\newdefinition{remark}{Remark}
\newdefinition{assumption}{Assumption}
\newproof{pf}{Proof}
\usepackage{mathtools,xparse}

\newcommand{\bq}{{q}}
\newcommand{\by}{{y}}
\newcommand{\brho}{{\rho}}

\newcommand{\blambda}{{\lambda}}
\newcommand{\bv}{{v}}
\newcommand{\bV}{{V}}
\newcommand{\bZ}{{ Z}}
\newcommand{\bx}{{x}}
\newcommand{\bX}{{X}}
\newcommand{\bz}{{z}}

\newcommand{\bPhi}{{ \Phi}}
\newcommand{\timp}{{t_{\rm eve}}}
\newcommand{\beforeimp}{|_{\timp}^{-}}
\newcommand{\afterimp}{|_{\timp}^{+}}
\newcommand{\atimp}{|_{\timp}}

\newcommand{\bg}{{g}}
\newcommand{\fun}{{ f}}

\newcommand{\Force}{\mathsf{F}}

\newcommand{\Faccel}{\mathsf{C}}
\newcommand{\bQ}{{ Q}}
\newcommand{\Mass}{\mathsf{M}}

\newcommand{\Rez}{\mathsf{R}}

\newcommand{\bzero}{\mathsf{0}}

\newcommand{\bI}{\mathsf{I}}

\newcommand{\Permutation}{\mathsf{P}}

\newcommand{\fin}{{\left(\,t,\,\bq,\,\bv,\,\brho \, \right)}}
\newcommand{\fincost}{\left(\,t,\,\bq,\,\bv,\,\dbv,\, \brho\,\right)}

\newcommand{\fincostdae}{\bigl(\,t,\,\bq(t,\brho),\,\bv(t,\brho),\,\dbv(t,\brho),\,\lambda(t,\brho),\,\brho\,\bigr)}

\newcommand{\fode}{{ \fun^{\scalebox{0.5}{\rm ode}}}} 
\newcommand{\fdae}{{ \fun^{\scalebox{0.4}{\rm DAE}}}}  
  
\newcommand{\fdaeimpv}{{ \fun^{\scalebox{0.4}{\rm DAE-imp-$\bv$}}}}  
\newcommand{\fdaeimpl}{{ \fun^{\scalebox{0.4}{\rm DAE-imp-$\lambda$}}}}  
\newcommand{\feom}{{ \fun^{\scalebox{0.5}{\rm eom}}}}  
\newcommand{\frho}{{\fun_\brho^{\scalebox{0.5}{\rm ode}}}}

\newcommand{\feomq}{{ \fun_\bq^{\scalebox{0.5}{\rm eom}}}}
\newcommand{\feomv}{{ \fun_\bv^{\scalebox{0.5}{\rm eom}}}}

\newcommand{\feomrhoi}{{\fun_{\brho_i}^{\scalebox{0.5}{\rm eom}}}}
\newcommand{\fdaeq}{{ \fun_\bq^{\scalebox{0.4}{\rm DAE}}}}
\newcommand{\fdaev}{{ \fun_\bv^{\scalebox{0.4}{\rm DAE}}}}
\newcommand{\fdaelb}{{ \fun^{\scalebox{0.4}{\rm DAE-{\scalebox{1.5}{$\lambda$}}}}}}
\newcommand{\fdaelbq}{{ \fun_\bq^{\scalebox{0.4}{\rm DAE-{\scalebox{1.5}{$\lambda$}}}}}}
\newcommand{\fdaelbv}{{ \fun_\bv^{\scalebox{0.4}{\rm DAE-{\scalebox{1.5}{$\lambda$}}}}}}
\newcommand{\fdaelbrhoi}{{ \fun_{\brho_i}^{\scalebox{0.4}{\rm DAE-{\scalebox{1.5}{$\lambda$}}}}}}
\newcommand{\fdaedv}{{ \fun^{\scalebox{0.4}{\rm DAE-{\scalebox{1.5}{$\dbv$}}}}}}
\newcommand{\fdaedvq}{{ \fun_\bq^{\scalebox{0.4}{\rm DAE-{\scalebox{1.5}{$\dbv$}}}}}}
\newcommand{\fdaedvv}{{ \fun_\bv^{\scalebox{0.4}{\rm DAE-{\scalebox{1.5}{$\dbv$}}}}}}
\newcommand{\fdaedvrhoi}{{ \fun_{\brho_i}^{\scalebox{0.4}{\rm DAE-{\scalebox{1.5}{$\dbv$}}}}}}
\newcommand{\fdaerho}{{ \fun_\brho^{\scalebox{0.4}{\rm DAE}}}}

\newcommand{\g}{\bg\fincost}
\newcommand{\gdae}{\bg\fincostdae}

\newcommand{\w}{w \bigl(t_F,\bq(t_F,\brho),\bv(t_F,\brho),\brho \bigr)}

\renewcommand{\Re}{{\mathds{R}}}

\newcommand{\dPhi}{{\dot \bPhi}}
\newcommand{\ddPhi}{{\ddot \bPhi}}
\newcommand{\dPhidq}{{\bPhi}_{\bq}}
\newcommand{\dPhidqq}{{\bPhi}_{\bq , \, \bq}}
\newcommand{\dPhidqdrho}{\bPhi_{\bq, \, \brho}}
\newcommand{\dPhidtdq}{\bPhi_{t, \, \bq }}

\newcommand{\dPhidqplus}{{{\bPhi}^+_{\bq}}}
\newcommand{\dPhidqqplus}{{{\bPhi}^+_{\bq , \, \bq}}}
\newcommand{\dPhidtdrhoplus}{{\bPhi^+_{t,\, \brho }}}
\newcommand{\dPhidqdrhoplus}{{\bPhi^+_{\bq, \, \brho}}}
\newcommand{\dPhidtdqplus}{{\bPhi^+_{t, \, \bq }}}
\newcommand{\dPhidtdvplus}{{\bPhi^+_{t, \,\bv }}}

\newcommand{\dtdPhidq}{{\dPhi}_{\bq}}
\newcommand{\frdrdrho}{\frac{dr}{d\brho}}

\newcommand{\frdtdrho}{\frac{d\timp}{d\brho}}
\newcommand{\Drho}{D_\brho}

\newcommand{\qrho}{\Drho\bq}
\newcommand{\dqrho}{\Drho\dbq}
\newcommand{\dtdrho}{{d\timp}/{d\brho}}

\newcommand{\dbq}{{\dot \bq}} 
 
\newcommand{\ddbq}{{\ddot \bq}} 
\newcommand{\dby}{{\dot \by}}
\newcommand{\dbv}{{\dot \bv}}
\newcommand{\dbZ}{{\dot \bZ}}
\newcommand{\dbz}{{\dot \bz}}
\newcommand{\bqq}{{{q},{q}}}
\newcommand{\dbV}{{\dot{\bV}}}

\newcommand{\qrhoplus}{{\bQ \afterimp}}
\newcommand{\qrhominus}{{\bQ \beforeimp}}

\newcommand{\dqrhoplus}{{\bV \afterimp}}
\newcommand{\dqrhominus}{{\bV \beforeimp}}

\newcommand{\Zplus}{{\bZ \afterimp}}
\newcommand{\Zminus}{{\bZ \beforeimp}}

\newcommand{\timpplus}{|_{\timp}^{+}}
\newcommand{\timpminus}{|_{\timp}^{-}}

\newcommand{\gplus}{\bg \timpplus}
\newcommand{\gminus}{\bg \timpminus}
\newcommand{\gplusarg}{\tilde{g} \big(\timp,\qtimp,\vplus, \brho\big)}
\newcommand{\gminusarg}{\tilde{g} \big(\timp,\qtimp,\vminus, \brho\big)}

\newcommand{\vplus}{{\bv \timpplus}}
\newcommand{\vminus}{\bv \timpminus}
\newcommand{\qplus}{{\bq \timpplus}}
\newcommand{\qminus}{\bq \timpminus}
\newcommand{\qtimp}{\bq|_{\timp}}
\newcommand{\vtimp}{\bv|_{\timp}}
\newcommand{\zplus}{{z \timpplus}}
\newcommand{\zminus}{z \timpminus}

\newcommand{\voneTauoneMinus}{{\bv_1}|_{\tau_1}^{-}}
\newcommand{\voneTauonePlus}{{\bv_1}|_{\tau_1}^{+}}

\newcommand{\voneTautwo}{{\bv_1}|_{\tau_2}}
\newcommand{\vtwoTauone}{{\bv_2}|_{\tau_1}}

\newcommand{\vtwoTautwoMinus}{{\bv_2}|_{\tau_2}^{-}}
\newcommand{\vtwoTautwoPlus}{{\bv_2}|_{\tau_2}^{+}}

\newcommand{\qoneTauone}{{\bq_1}|_{\tau_1}}
\newcommand{\qoneTautwo}{{\bq_1}|_{\tau_2}}
\newcommand{\qtwoTauone}{{\bq_2}|_{\tau_1}}
\newcommand{\qtwoTautwo}{{\bq_2}|_{\tau_2}}

\newcommand{\zoneTauone}{{\bz_1}|_{\tau_1}}
\newcommand{\zoneTautwo}{{\bz_1}|_{\tau_2}}
\newcommand{\ztwoTauone}{{\bz_2}|_{\tau_1}}
\newcommand{\ztwoTautwo}{{\bz_2}|_{\tau_2}}

\newcommand{\eps}{\varepsilon}

\newcommand{\zdytwo}{\bz = \int_{t_0}^{t_F} \dot y_2 \ {\rm dt}}
\newcommand{\zddytwo}{\bz = \int_{t_0}^{t_F} \ddot y_2 \ {\rm dt}}

\newcommand{\pttwo}{ y_2}
\newcommand{\dpttwo}{\dot y_2}

\nomenclature[00]{$\feom$}{The function solving the Equation of Motion}%
\nomenclature[00]{$\fode$}{The function solving the ODE system}%

\nomenclature[01]{$\bq$}{The generalized position state vector}%
\nomenclature[01]{$\dbq$}{The generalized velocity state vector}%
\nomenclature[01]{$\ddbq$}{The generalized acceleration state vector}%
\nomenclature[01]{$\y$}{The reduced-order state vector} %
\nomenclature[01]{$t_i$}{The time of impact}%

\nomenclature[02]{$\bQ$}{The generalized force vector}%
\nomenclature[02]{$\bM$}{The generalized mass matrix}%
\nomenclature[02]{$\hat{\bQ} $}{The generalized force vector}%
\nomenclature[02]{$\hat{\bM}$}{The generalized mass matrix}%

\nomenclature[03]{$\bPhi$}{The equation of constraints}%
\nomenclature[03]{$\dPhidq$}{The Jacobian of the equation of constraints}%
\nomenclature[03]{$\dPhidqq$}{The three-dimensional Jacobian of the equation of constraints}%

\nomenclature[04]{$\yrho$}{The sensitivity matrix of the $\y$ state vector with respect to the parameters $\brho$ }%
\nomenclature[04]{$\dtdrho$}{The sensitivity of the time of impact ${\bf t_i}$ with respect to the parameters $\brho$ }%
\nomenclature[04]{$\qrho$}{The sensitivity matrix of the $\bq$ state vector with respect to the parameters $\brho$ }%
\nomenclature[04]{$\dqrho$}{The sensitivity matrix of the $\dbq$ state vector with respect to the parameters $\brho$ }%

\nomenclature[05]{$\frho$}{The partial derivative of the ODE function with respect to the parameters $\brho$ }%
\nomenclature[05]{$\fy$}{The partial derivative of the ODE function with respect to the state vector $\y$ }%

\nomenclature[06]{$\nabla_{\rho}(.) $}{Gradient of a function}%
\nomenclature[06]{$\dot{(.)} $}{Time derivative of a function}%
\nomenclature[06]{$\psi(t)$}{The cost function}%
\nomenclature[06]{$g$}{The ruining cost function}%
\nomenclature[06]{$z$}{The quadrature variable}%
\nomenclature[06]{$Z$}{The sensitivity of the quadrature variable}%
\nomenclature[06]{$w$}{The terminal cost function}%

\nomenclature[07]{$\y^{-}$}{Value of the state vector $\y$ before impulse at the time of impact}%
\nomenclature[07]{$\y^{+}$}{Value of the state vector $\y$ after impulse at the time of impact}%
\nomenclature[07]{$\bq^{-}$}{Value of the generalized position state vector $\bq$ before impulse at the time of impact}%
\nomenclature[07]{$\bq^{+}$}{Value of the generalized position state vector $\bq$ after impulse at the time of impact}%
\nomenclature[07]{$\gplus$}{ruining cost function after impulse at the time of impact}%
\nomenclature[07]{$\gminus$}{ruining cost function before impulse at the time of impact}%

\nomenclature[08]{\overline{\bf M}}{Mass matrix for the penalty formulation}

\journal{Journal Name}

\begin{document}
	
\thispagestyle{empty}
\setcounter{page}{0}

\makeatletter
\def\Year#1{%
  \def\yy@##1##2##3##4;{##3##4}%
  \expandafter\yy@#1;
}
\makeatother

\begin{Huge}
\begin{center}
Computational Science Laboratory Technical Report CSL-TR-\Year{\the\year}-{\tt 7} \\
\today
\end{center}
\end{Huge}
\vfil
\begin{huge}
\begin{center}
Sebastien Corner, Corina Sandu, and Adrian Sandu
\end{center}
\end{huge}

\vfil
\begin{huge}
\begin{it}
\begin{center}
``{\tt Modeling and sensitivity analysis methodology for hybrid dynamical systems}''
\end{center}
\end{it}
\end{huge}
\vfil

\begin{large}
\begin{center}
Computational Science Laboratory \\
Computer Science Department \\
Virginia Polytechnic Institute and State University \\
Blacksburg, VA 24060 \\
Phone: (540)-231-2193 \\
Fax: (540)-231-6075 \\ 
Email: \url{scorner@vt.edu} \\
Web: \url{http://csl.cs.vt.edu}
\end{center}
\end{large}

\vspace*{1cm}

\begin{tabular}{ccc}
\includegraphics[width=2.5in]{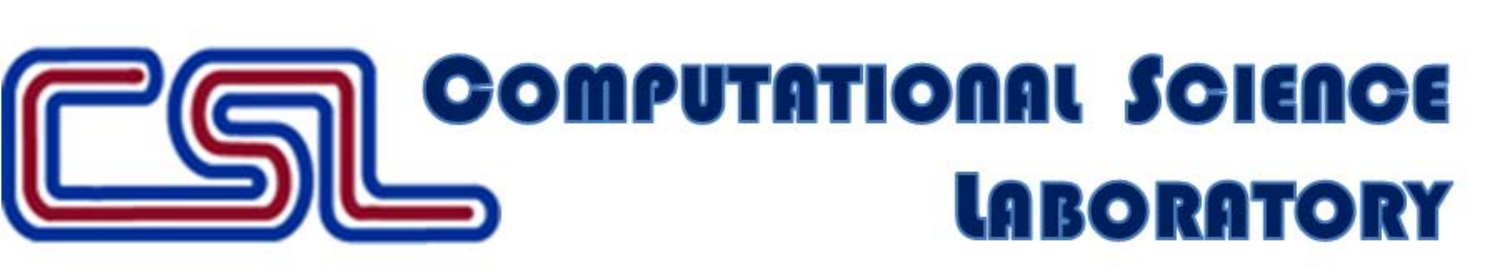}
&\hspace{2.5in}&
\includegraphics[width=2.5in]{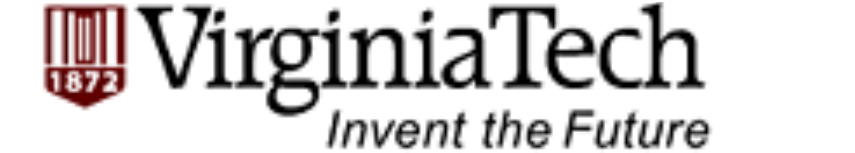} \\
{\bf\large \textit{Compute the Future}} &&\\
\end{tabular}

\newpage

\begin{frontmatter}

\title{Modeling and sensitivity analysis methodology for hybrid dynamical systems} 


\author[rvt,focal]{Sebastien Corner}
\author[rvt]{Corina Sandu}
\ead{csandu@vt.edu, 
107 Randolph Hall – 0710
460 Old Turner Street
Blacksburg, VA 24061}
\author[focal]{Adrian Sandu}

\address[rvt]{Advanced Vehicle Dynamics Laboratory,
Department of Mechanical Engineering,
Virginia Tech,
Blacksburg, VA 24061}
\address[focal]{Computational Science Laboratory,
Department of Computer Science,
Virginia Tech,
Blacksburg, VA 24061}
\begin{abstract}
This paper provides  an analytical methodology to compute the sensitivities with respect to system parameters for any second order hybrid Ordinary Differential Equation (ODE) system. The hybrid ODE system is characterized by discontinuities in the velocity state variables due to an impulsive jump caused by an instantaneous impact in the motion of the system. The analytical methodology that solves this problem is structured based on jumping conditions for both the state variables and the sensitivities matrix. The proposed analytical approach is of the benchmarked against a numerical method.
\end{abstract}
\end{frontmatter}



\doublespacing
\section{Introduction}
Sensitivity analysis plays a key role in a wide range of computational engineering problems such as design optimization, optimal control, and implicit time integration methods, by providing derivative information for gradient based algorithms and methods. Sensitivity analysis quantifies the effect of small changes in the system parameters onto the outputs of interest ~\cite{Sandu2015dynamic}. Specifically, in the design of mechanical systems sensitivity analysis reveals the system parameters that affect the given performance criterion the most, thus providing directions for mechanical design improvements. Sensitivity analysis enables gradient-based optimization by providing the derivative of the cost function with respect to design variables.
In adaptive control systems sensitivity analysis allows to assess  stability by accounting for the effects of system disturbances and system parameters inaccuracies. 

The most widely used techniques for sensitivity analysis are the direct and the adjoint methods.  These approaches are complementary, as the direct sensitivity provides information on how parametric uncertainties propagate through the system dynamics, while the adjoint method is suitable for inverse modeling, in the sense that it can be used to identify the origin of uncertainty in a given model output ~\cite{sandu2003direct}.

Numerical approximations of sensitivities are often calculated by finite difference methods where the deviation of the state trajectories are evaluated after system parameters disturbances or variations in the initial conditions are added to the system. Because of the simplicity of this method, which does not require any additional inputs other than the provided model, this technique is broadly used. However, the accuracy of the results is severely limited by the perturbation size, and by the roundoff and cancellation errors ~\cite{chang2016design}.

This paper develops a general and unify formulation for {\it direct} sensitivity analysis for hybrid dynamical systems.
In the context of this study, the term hybrid refers to a continuous system that encounters a finite number of events where some of the state variables jump to different values; the dynamics of a hybrid system is piecewise-smooth in time. Sensitivity analysis for hybrid systems has been an active topic of research ~\cite{Barton1998,Barton1999, barton2002modeling, tolsma2002hidden, rozenvasser1967general, saccon2014sensitivity, hiskens2006sensitivity, hiskens2000trajectory, taringoo2010geometry}.

We are especially interested in hybrid mechanical systems where sensitivity analysis involves the time, the position coordinates, the velocity coordinates, and the system parameters. We treat unconstrained systems modeled by ODEs, as well as constrained multibody systems modeled by differential algebraic equations (DAEs) and ODE penalty formulation ~\cite{Sandu2015dynamic,kurdila1993sufficient}. The sensitivity of the time of event, and the jumps in the sensitivities of the state variables at the time of event, are available in the literature. Here a new graphical proof of the jumps in sensitivities at the time of an event is employed, which helps to better understand the conditions for the jump in the sensitivities.  This paper provides a unified methodology for determining the system solutions, their sensitivities, and sensitivities of a cost function for different types of events.  The first type of event is caused by an external impulse (e.g., a contact) leading to a sudden change of velocities. The second type of event is caused by a sudden change of the equations of motion.  The third type of event is caused by  a sudden change in the kinematic constraints.  

The paper is organized as follows. A review of the direct sensitivity approach for smooth ODE systems, along with the quadrature variable of the running cost function, is introduced in Section \ref{sec:Direct sensitivity analysis for smooth ODE systems}.  The  extension of this approach to hybrid ODE systems is presented in Section \ref{sec:direct-sensitivity-hybrid}. The sensitivity analysis for smooth constrained rigid multibody dynamic systems is reviewed in Section \ref{sec:multibody-smooth}.  Sensitivity analysis methodology for hybrid constrained systems is developed in Section \ref{sec:multibody-constrained}.
Sensitivity analysis methodology for constrained systems with a sudden change of the of equation of motions is developed in Section \ref{sec:Direct sensitivity analysis with transition functions}. The proposed  sensitivity analysis methodologies in Section \ref{sec:multibody-constrained} are applied to a five-bar mechanism  in Section \ref{sec:numerics}. Conclusions are drawn in Section \ref{sec:conclusions}.

\section{Direct sensitivity analysis for smooth ODE systems}
\label{sec:Direct sensitivity analysis for smooth ODE systems}

We start the discussion with a review of direct sensitivity analysis for dynamical systems governed by smooth ODEs.
\subsection{Smooth ODE system dynamics}

In this study we consider second order systems of ordinary differential equations of the form:
\begin{equation}
\label{eq:EOM}
{\Mass}\left(t,\bq,\brho\right) \cdot \ddbq ={\Force} \left(t,\bq,\dbq,\brho\right),
\quad t_0 \leq t \leq  t_F, \quad \bq(t_0)=\bq_0(\brho),\quad\dbq(t_0)=\dbq_0(\brho),
\end{equation}
or equivalently:
\begin{equation}
\label{eq:EOM-ODE}
\ddbq ={\Mass}^{-1} \left(t,\bq,\brho\right)  \cdot {\Force}\left(t,\bq,\dbq,\brho\right)
=: \feom \left(t,\bq,\dbq,\brho\right),
\end{equation}
that arise from the description of the dynamics of mechanical systems. In \eqref{eq:EOM-ODE} $t \in \mathds{R}$ is time, $\bq \in \mathds{R}^{n}$ is the generalized position vector and $\dbq \in \mathds{R}^{n}$ is the generalized velocity vector, $n$ is the dimension of generalized coordinates, and $\brho \in \mathds{R}^{p} $ is the vector of system parameters, where $p$ is the number of parameters. The dot notation ($\dot{\square}$ or $\ddot{\square}$) indicates the total (first or second order) derivative of a function or variable with respect to time. Subscripts indicate partial derivative with respect to a quantity, unless stated otherwise. The mass matrix $\Mass:\mathds{R} \times \mathds{R}^{n} \times \mathds{R}^{p} \rightarrow \mathds{R}^{n \times n}$ is assumed to be smooth with respect to all its arguments, invertible, and with an inverse $\Mass^{-1}$ that is also smooth with respect to all arguments. The right hand side function $\Force:\mathds{R} \times \mathds{R}^{n} \times \mathds{R}^{n} \times \mathds{R}^{p} \rightarrow \mathds{R}^{n}$ represents external and internal generalized forces and is assumed to be smooth with respect to all its arguments.

The state trajectories are obtained by integrating the equations of motion \eqref{eq:EOM-ODE}, which depend on the system parameters $\brho$. Consequently, the state trajectories (the solutions of the equations of motion) depend implicitly on time and on the parameters, $\bq=\bq(t, \brho)$ and $\dbq=\dbq(t, \brho)$. The state trajectories also depend implicitly on the initial conditions of \eqref{eq:EOM-ODE}. For clarity we denote the velocity state variables by $\bv=\dbq \in \mathds{R}^{n}$.

Sensitivity analysis computes derivatives of the solutions of \eqref{eq:EOM-ODE} with respect to  the system parameters:
\begin{equation}
\label{eq:sensitivity-of-solutions}
\bQ(t,\brho) := \Drho\bq(t) := \frac{d \bq}{d \brho}(t, \brho) \in \mathds{R}^{n \times p}, 
\quad \bV(t,\brho) := \Drho\bv(t) = \frac{d \bv}{d \brho}(t, \brho) \equiv \frac{d \dbq}{d \brho}(t, \brho)=\dot{\bQ}(t,\brho)  \in \mathds{R}^{n \times p}.
\end{equation}

The second order ODE~\eqref{eq:EOM-ODE} can transformed into a first order reduced system as follows. With the velocity state variables $\bv:=\dbq \in \mathds{R}^{n}$ the system \eqref{eq:EOM-ODE} can be written in the form:
\begin{equation}
\label{eq:EOM-ODE-first-order}
\begin{bmatrix}
\bI & \bzero \\
\bzero & \Mass(t,\bq,\brho)
\end{bmatrix}
\begin{bmatrix}
 \dbq \\
 \dbv
\end{bmatrix}  = 
\begin{bmatrix}
 \bv \\
 \Force(t,\bq,\bv,\brho)
\end{bmatrix}
\quad \Leftrightarrow \quad
\begin{bmatrix}
 \dbq \\
 \dbv
\end{bmatrix}
=
\begin{bmatrix}
 \bv \\
\feom \left(t,\bq,\bv,\brho\right)
\end{bmatrix},
\qquad
\begin{bmatrix}
 \bq(t_0) \\
 \bv(t_0)
\end{bmatrix}
=
\begin{bmatrix}
 \bq_0(\brho) \\
 \bv_0(\brho)
\end{bmatrix}.
\end{equation}
%
%
\begin{definition}[Cost function] 
\label{def:CostFunction}
Consider a smooth `trajectory cost function' $\bg:\mathds{R}^{1+3n+p} \rightarrow \mathds{R}$ and a smooth `terminal cost function' $w:\mathds{R} \times \mathds{R}^{1+2n+p} \rightarrow \mathds{R}$. 
A general cost function is defined as the sum of the costs along the trajectory plus the cost at the terminal point of the solution:
\begin{eqnarray}
\label{eq:CostFunction}
&&\psi=  \int_{t_0}^{t_F} {\g \  dt} + \w.
\end{eqnarray}
\end{definition}
\begin{remark}[Accelerations in the cost function]
\label{rem:Accelerations}
Note that the trajectory cost function \eqref{eq:CostFunction} includes accelerations via $\dbv$. Accelerations are not independent variables and they can be resolved in terms of positions and velocities:
\[
g\left(\,t,\,\bq,\,\bv,\,\dbv,\,\brho\,\right) = g\bigl(\,t,\,\bq,\,\bv,\,\feom(t,\bq,\bv,\brho),\,\brho\,\bigr) = \tilde{g}\bigl(\,t,\,\bq,\,\bv,\,\brho\,\bigr).
\] 
We prefer to keep accelerations as an explicit argument in the cost function \eqref{eq:CostFunction} in order to give additional flexibility in practical applications. However, we will need to resolve the sensitivities of acceleration in terms of other sensitivities in subsequent calculations.
\end{remark}

To further simplify the notation we define the `quadrature' variable $z(t) \in \mathds{R}$ as follows:
\begin{equation} 
\label{eqn:quadrature-variable}
\begin{split}
\bz(t,\brho) &:=  \int_{t_0}^{t} { \tilde{g}\bigl(t,\bq(t,\brho),\bv(t,\brho),\brho\bigr) \,  dt} \quad \Leftrightarrow \quad \\
\dbz(t,\brho) &=  \g = \tilde{g}\bigl(t,\bq(t,\brho),\bv(t,\brho),\brho\bigr),  \quad t_0 \leq t \leq  t_F, \quad z(t_0,\brho)=0.
\end{split}
\end{equation}
The cost function \eqref{eq:CostFunction} reads:
\begin{equation}
\label{eqn:quadrature-cost-function}
\psi =\bz(t_F) + \w.
\end{equation}
\begin{definition}[The canonical ODE system] 
\label{def:canonical-ode}
The canonical  system is obtained by combining the first order ODE dynamics \eqref{eq:EOM-ODE-first-order}  with equation \eqref{eqn:quadrature-variable} for the `quadrature' variable':
%
\begin{subequations}
\label{eq:canonical-ode}
\begin{eqnarray}
\label{eq:canonical-ode-system}
\bx(t) := \begin{bmatrix}  {\bq}(t) \\ {\bv}(t)  \\ \bz(t)   \end{bmatrix}  \in \mathds{R}^{2 n+1}; 
\qquad
\dot{\bx} 
= \begin{bmatrix} \dbv \\ \feom\bigl(t,\bq(t,\brho),\bv(t,\brho),\brho\bigr)  \\ \tilde{g}\fin \end{bmatrix} = F(t,\bx,\brho), \quad t_0 \le t \le t_F,
\quad
\bx(t_0,\brho) = \begin{bmatrix} \bq_0(\brho) \\ \bv_0(\brho) \\ 0 \end{bmatrix}.
\end{eqnarray}
The canonical cost function \eqref{eqn:quadrature-cost-function} is purely a terminal cost function:
\begin{equation}
\label{eqn:canonical-cost-function}
\psi =\bz(t_F,\brho) + \w = W\bigl( x(t_F,\brho), \brho \bigr).
\end{equation}
\end{subequations}
\end{definition}
%

\subsection{Direct sensitivity approach for smooth ODE systems}
\label{sec:direct-sensitivity-continuous}
%

\begin{definition}[The sensitivity analysis problem] 
\label{def:Sensitivity}
Our goal in this work is to perform a sensitivity analysis of the cost function, i.e., to compute the total derivative of the cost function \eqref{eq:CostFunction} with respect to model parameters $\brho$:
\begin{equation}
\label{eq:gradient-cost-function}
\Drho \psi = \frac{d\, \psi}{d\, \brho}  \in \mathds{R}^{1 \times p}.
\end{equation}
Note that the cost function \eqref{eq:CostFunction} depends on the system parameters $\brho$ directly (through the direct dependency of $g$ and $w$ on $\brho$) as well as indirectly (through the dependency of $\bq$ and $\bv$ on $\brho$). The sensitivity \eqref{eq:gradient-cost-function} needs to account for all the direct and the indirect dependencies. 
\end{definition}

\begin{remark}[The direct sensitivity analysis approach] 
In order to compute the sensitivities \eqref{eq:gradient-cost-function} we take a variational calculus approach \cite{Sandu_2014_sensitivity_ODE_multibody,Sandu2015dynamic,Sandu_2017_vehicle-optimization}.  Infinitesimal changes in parameters 
\[
\brho \to \brho + \delta\brho \in \Re^p,
\]
lead to a total change in the cost function as follows:
\[
 \psi \to \psi + \delta \psi, \quad
\delta\psi = \sum_{i=1}^p \frac{d \psi}{d \brho_i} \cdot \delta\brho_i =
\Drho \psi \cdot \delta\brho. 
\] 
The direct sensitivity analysis computes each element $(\Drho \psi)_i = \partial \psi/\partial \brho_i$ of the derivative \eqref{eq:gradient-cost-function} by accounting for changes in the cost function that result from changing each individual parameter $\delta\brho_i$, $i=1 \dots p$.  
\end{remark}

\subsection{Direct sensitivity analysis with respect to system parameters solved analytically}
\begin{definition}[The tangent linear model (TLM)] Consider the `position sensitivity' matrix $Q(t,\brho)$ and the `velocity sensitivity' matrix  $V(t,\brho)$ defined in \eqref{eq:sensitivity-of-solutions}:
\begin{subequations}
\label{eq:sensitivity-matrices}
\begin{eqnarray}
\label{eq:sensitivity-matrix-Q}
\bQ_i(t,\brho) &:= \frac{d\, \bq(t,\brho)}{d\, \brho_i} \in \mathds{R}^{n}, ~~ i=1,\dots,p; \quad
\bQ(t,\brho) := \begin{bmatrix} \bQ_1(t,\brho) \cdots \bQ_p(t,\brho) \end{bmatrix} \in \mathds{R}^{n \times p}, \\
\label{eq:sensitivity-matrix-V}
\bV_i(t,\brho) &:= \frac{d\, \bv(t,\brho)}{d\, \brho_i} \in \mathds{R}^{n}, ~~ i=1,\dots,p; \quad
\bV(t,\brho) := \begin{bmatrix} \bV_1(t,\brho) \cdots \bV_p(t,\brho) \end{bmatrix} \in \mathds{R}^{n \times p}.
\end{eqnarray}
\end{subequations}
These sensitivities evolve in time according to the {\it tangent linear model} (TLM) equations \cite{Sandu_2014_sensitivity_ODE_multibody,Sandu2015dynamic,Sandu_2017_vehicle-optimization}, obtained by differentiating the equations of motion \eqref{eq:EOM-ODE} with respect to the parameters:
\begin{subequations}
\label{eq:TLM}
\begin{equation}
\begin{cases}
\dot{\bQ}_i &\displaystyle = \frac{d\dbq}{d \brho_i} = \frac{d\bv}{d \brho_i} = \bV_i, \\[8pt]
\dot{\bV}_i &\displaystyle =  \frac{d\dbv}{d \brho_i} = \feomq \left(t,\bq,\bv,\brho\right) \cdot \frac{d\bq}{d \brho_i}
+ \feomv \left(t,\bq,\bv,\brho\right) \cdot \frac{d\bv}{d \brho_i}
+ \feomrhoi \left(t,\bq,\bv,\brho\right) \\
&=  \feomq \left(t,\bq,\bv,\brho\right) \cdot \bQ_i
+ \feomv \left(t,\bq,\bv,\brho\right) \cdot \bV_i
+ \feomrhoi \left(t,\bq,\bv,\brho\right),
\end{cases}
\qquad i=1,\dots,p,
\quad t_0 \le t \le t_F,
\end{equation}
with the initial conditions
\begin{equation}
\label{eq:TLM-IC}
Q_i(t_0,\brho) = \frac{d\bq_0}{d \brho_i}, \quad
V_i(t_0,\brho) = \frac{d\bv_0}{d \brho_i}, \quad
i=1,\dots,p.
\end{equation}
\end{subequations}
The expressions $\feomq$, $\feomv$, and $\feomrhoi$ denote the partial derivatives of $\feom$ with respect to the subscripted variables.
\end{definition}

\begin{remark}
The partial derivatives $\partial \feom/\partial \zeta$ are obtained by differentiating  \eqref{eq:EOM-ODE} with respect to $\zeta \in \{ \bq,\bv,\brho \}$:
\begin{equation}
\label{eq:EOM-ODE}
\frac{\partial \feom}{\partial \zeta} 
=\frac{\partial (\Mass^{-1}  \, \Force)}{\partial \zeta} 
= -\Mass^{-1} \Mass_\zeta \Mass^{-1}\,\Force
+ \Mass^{-1}\,\Force_\zeta 
= \Mass^{-1}\,\left( \Force_\zeta - \Mass_\zeta\,\feom \right)
= \Mass^{-1}\,\left( \Force_\zeta - \Mass_\zeta\,\dbv \right).
\end{equation}
%
\end{remark}

\begin{definition}[The quadrature sensitivity] 
\label{The quadrature sensitivity}
Similarly, let the `quadrature sensitivity' vector $Z(t,\brho)$ be the Jacobian of the `quadrature' variable $z(t,\brho)$ ~\eqref{eqn:quadrature-variable} with respect to the parameters $\brho$:
\[
Z_i(t,\brho) := \frac{\partial z(t,\brho)}{\partial \brho_i}, ~~ i=1,\dots,p; \quad
Z(t,\brho) := \nabla_\brho z(t,\brho) = \begin{bmatrix} Z_1(t,\brho) \cdots Z_p(t,\brho) \end{bmatrix} \in \mathds{R}^{1 \times p}.
\]
The time evolution equations of the quadrature sensitivities are given by the TLM obtained by differentiating \eqref{eqn:quadrature-variable} with respect to the parameters:
\begin{equation}
\label{eq:TLM-quadrature0}
\begin{split}
\dbZ_i  =  \frac{d\, g\fincost}{d\,\brho_i} &= g_\bq\cdot \bQ_{i} +  g_\bv\cdot \bV_{i}  + g_{\dbv}\cdot\dot{\bV}_{i} + g_{\brho_i} \\
&= \big(g_\bq +  g_\dbv\,\feomq\big)\cdot \bQ_{i}  
+  \big(g_\bv +  g_\dbv\,\feomv\big)\cdot \bV_{i}+ g_{\brho_i} +  g_\dbv\cdot\feomrhoi, \\
& \quad t_0 \le t \le t_F, \quad  \bZ_i(t_0,\brho) = 0, \quad i=1,\dots,p.
\end{split}
\end{equation}
\end{definition}
\begin{definition} [Canonical sensitivity ODE]
\label{def:Canonical sensitivity ODE}
The solutions given by Eq.~\eqref{eq:canonical-ode-system}, the TLM given by Eq.~\eqref{eq:TLM}, and the sensitivity quadrature equations \eqref{eq:TLM-quadrature} need to be solved together forward in time, leading to the canonical sensitivity ODE that computes the derivatives of the cost function with respect to the system parameters $\brho$ for smooth systems:
\begin{eqnarray}
\label{eq:canonical-ode-sensitivity}
\begin{bmatrix}
\dbq \\ \dbv \\ \dbz \\  \bigl[\dot{\bQ}_i \bigr]_{i=1,\dots,p}\\
 \bigl[\dot{\bV}_i \bigr]_{i=1,\dots,p}\\  \bigl[\dbZ_i\bigr]_{i=1,\dots,p}
\end{bmatrix} 
=
\begin{bmatrix}
\bv \\
\feom\\
 \tilde{\bg}\\
 \bigl[\bV_i \bigr]_{i=1,\dots,p}\\
 \bigl[\feomq \bQ_i +
\feomv \bV_i +
\feomrhoi \bigr]_{i=1,\dots,p} \\
 \bigl[\big(g_\bq +  g_\dbv\,\feomq\big)\cdot \bQ_{i}  
+  \big(g_\bv +  g_\dbv\,\feomv\big)\cdot \bV_{i}+ g_{\brho_i} +  g_\dbv\cdot\feomrhoi \bigr]_{i=1,\dots,p}
\end{bmatrix},
\end{eqnarray}
where the state vector of the canonical sensitivity ODE is :
\begin{eqnarray}
\label{eq:canonical-sensitivity-state}
\bX = \left[   \, \bq^{\rm T}, \, \bv^{\rm T}, \, \bz, ~
\bQ_1^{\rm T},  \dots , \bQ_p^{\rm T}, \,
\bV_1^{\rm T},  \dots , \bV_p^{\rm T}, \,
\bZ_1, \dots, \bZ_p
\right] ^{\rm T}  \in \mathds{R}^{(n+1)(p+1)}.
\end{eqnarray}
\end{definition}

\begin{remark}  [The sensitivities of the cost function]
\label{rem:sensitivities-of-the-cost-function}
Once the quadrature sensitivities \eqref{eq:TLM-quadrature0} have been calculated, the sensitivities of the cost function with respect to each parameter are computed as follows:
\[
\begin{split}
\frac{d\,\psi}{d\, \brho_i} &= \bZ_i(t_F) + w_\bq\bigl(\,t_F, \,\bq(t_F,\brho),\, \bv(t_F,\brho),\, \brho\,\bigr)\cdot Q_i(t_F,\brho) 
+ w_\bv\bigl(\,t_F, \,\bq(t_F,\brho),\, \bv(t_F,\brho),\, \brho\,\bigr)\cdot V_i(t_F,\brho) \\
& \qquad + w_{\brho_i}\bigl(\,t_F, \,\bq(t_F,\brho),\, \bv(t_F,\brho),\, \brho\,\bigr),  \qquad i=1,\dots,p. 
\end{split}
\]
\end{remark}

\subsection{Direct sensitivity analysis with respect to system parameters solved with the complex finite difference method}
\label{complex fdf}
An accurate numerical method for sensitivity analysis  of a smooth ODE system with respect to the system parameters $\brho$ is the complex finite difference method \cite{Sandu_2014_sensitivity_ODE_multibody,Sandu2015dynamic,Sandu_2017_vehicle-optimization}. Add a small {\it complex} perturbation to one parameter:
\[
\tilde{\brho}_j = \begin{cases}
\brho_j  & \textnormal{for}~~j \ne \ell, \\
\brho_\ell + \mathfrak{i}\,\Delta\brho& \textnormal{for}~~j = \ell,
\end{cases}
\quad j=1,\dots,p,
\]
and solve the canonical ODE system \eqref{eq:canonical-ode-system} for this perturbed values of the parameters to obtain:
\[
\bq(t,\tilde{\brho}), \quad \bv(t,\tilde{\brho}), \quad \bz(t,\tilde{\brho}), 
\quad \psi(\tilde{\brho})=\bz(t_F,\tilde{\brho}) + w \bigl(\, \bq(t_F,\tilde{\brho}),\, \bv(t_F,\tilde{\brho}),\,\tilde{\brho}\, \bigr).
\]
The sensitivities are approximated numerically by the imaginary parts of the state variables:
\begin{equation}
\label{eq:ODE_sensitivity_complex0}
\bQ_\ell(t,\brho)  \approx - \frac{\operatorname{imag}\bigl({\bq(t,\tilde{\brho})}\bigr)}{\Vert{\Delta \brho}\Vert}, \quad
\bV_\ell(t,\brho)  \approx - \frac{\operatorname{imag} \bigl(\bv(t,\tilde{\brho}) \bigr)}{\Vert{\Delta \brho}\Vert}, \quad
\frac{d\, \psi}{d\, \brho_\ell}  \approx - \frac{\operatorname{imag} \bigl(\psi(\tilde{\brho}) \bigr)}{\Vert{\Delta \brho}\Vert}.
\end{equation}

We next discuss an approach to sensitivity analysis that accounts for discontinuities in the state variables.

\section{Direct sensitivity analysis for hybrid ODE systems}
\label{sec:direct-sensitivity-hybrid}
%
\subsection{Hybrid ODE systems}
\begin{definition}[Hybrid dynamics] 
\label{def:time-of-invent}
A hybrid mechanical system is  a piecewise-in-time continuous dynamic ODE described by \eqref{eq:EOM-ODE} that exhibits discontinuous dynamic behavior in the generalized velocity state vector at a finite number of time moments (no zeno phenomena \cite{Pace2017}).  Each such moment is a `time of event' $\timp$ and corresponds to a triggering event described by the equation:
\begin{equation}
\label{eq:event_function}
r \big( \qtimp \big) = 0,
\end{equation} 
where $r : \mathds{R}^{n} \to \mathds{R}$ is a smooth `event function'.  
\end{definition}

\begin{remark}
In the context of this paper, we assume that there are no grazing phenomena where the system trajectory would make tangential contact with an the event triggering hypersurface \cite{Pace2017}.
\end{remark}

\begin{definition}[Characterization of an event]
\label{def:characterize-invent} 
For hybrid systems variables can change values during the event. For this reason we distinguish between the value of a variable right before the event $\boldsymbol{x}\beforeimp$, and its value right after the event $\boldsymbol{x}\beforeimp$:
\[
\boldsymbol{x}\beforeimp := \lim_{\varepsilon>0,~\varepsilon\to 0}\,\boldsymbol{x}(\timp-\varepsilon), \quad
\boldsymbol{x}\afterimp := \lim_{\varepsilon>0,~\varepsilon\to 0}\,\boldsymbol{x}(\timp+\varepsilon).
\]
The limits exist since the evolution of the system is smooth in time both before and after the event.

We consider an event happening at time $\timp$ that applies a finite energy impulse force to the system. Such an impulse force does not change the generalized position state variables, and therefore:
\begin{equation}
\label{eq:jump-in-position}
\qplus = \qminus = \bq\atimp.
\end{equation}
However, the finite energy event can abruptly change the generalized velocity state vector $\dbq$ from its value $\vminus$ right before the event to a new value $\vplus $ right after the event.  The `jump function' at the time of event $\timp$ characterizes the change in the generalized velocity during the event: 
\begin{equation}
\label{eq:jump-in-velocity}
\vplus=h \Big({\timp},\, \bq\atimp,\,\vminus,\,{\brho} \Big)  \quad \Leftrightarrow \quad  \dbq \afterimp=h \Big({\timp},\, \bq\atimp,\,\dbq  \beforeimp,\,{\brho} \Big) .
\end{equation}
\end{definition} 

\begin{remark}[Multiple events]
In many cases the change can be triggered by one of multiple events. Each individual event is described by the event function $r_\ell : \mathds{R}^n \to \mathds{R}$, $\ell=1,\dots,e$. The detection of the next event, which can be one of the possible $e$ options, is described by:
\[
r_1 \big(  \qtimp \big) \cdot r_2 \big( \qtimp \big)~ \dots ~ r_e \big(  \qtimp\big)= 0,
\]
and if event $\ell$ takes place then $r_\ell = 0$ and the corresponding jump is:
\[
\vplus = h_\ell \Big(\qtimp,\,\vminus\Big).
\]
\end{remark}

\begin{remark}[Numerical implementation of invents]
Numerical solutions of hybrid systems use an event detection mechanism. The event function \eqref{eq:event_function} is implemented in the numerical time solver such that the integrator is stopped at the solution of \eqref{eq:event_function}. The jump function \eqref{eq:jump-in-velocity} is implemented as a callback function that is executed after the event is detected. The numerical integration resumes with new initial conditions after the jump.
\end{remark}

\begin{definition}[Twin perturbed systems]
\label{def:perturbed-twins}
Consider two versions of the system \eqref{eq:EOM-ODE} with identical dynamics and initial conditions, but with different parameters values $\brho_1$ and $\brho_2$, respectively. Without loss of generality in this proof we consider the scalar parameter case $p=1$; the general equation \eqref{eq:jump-position-sensitivity} can be proven element by element by considering sensitivities with respect to individual parameters. The two parameters are infinitesimally small perturbations $\delta\brho$ of the reference parameter value $\brho$:
\[
\brho_1 = \brho - \frac{\delta\brho}{2}; \qquad  \brho_2 = \brho + \frac{\delta\brho}{2}. 
\]
We denote by $\bq_1(t) = \bq(t,\brho_1)$, $\bv_1(t) = \bv(t,\brho_1)$, and $\bz_1(t) = \bz(t,\brho_1)$ the position and velocity states, and the quadrature variable of the first system, respectively. We denote by $\bq_2(t) = \bq(t,\brho_2)$, $\bv_2(t) = \bv(t,\brho_2)$ , and $\bz_2(t) = \bz(t,\brho_2)$ the position and velocity states, and the quadrature variable of the second system, respectively.

Assume that the sign of the perturbation $\delta\brho$ is such that $\tau_{2} > \timp > \tau_{1}$, and denote $\delta\tau = \tau_{2} - \tau_{1}$. Since $\delta\brho$ is infinitesimally small, so is $\delta\tau$. The trajectories of the positions $\bq(t)$, $ \bq_1(t)$, and $\bq_2(t)$, as well as the trajectories of the velocities $\bv(t)$, $\bv_1(t)$, and $ \bv_2(t)$, are schematically illustrated in Fig.~\ref{fig:im1} and Fig.~\ref{fig:im2}, respectively. As shown in Fig.~\ref{fig:im1} the first system meets the event described by the function \eqref{eq:event_function} at the time of event $\timp(\brho_1) = \tau_{1}$, when its position state is $\qoneTauone$. The second system meets the event at time $\timp(\brho_2) =\tau_{2}$, when its position state is $\qtwoTautwo$. Note that in the limit of vanishing $\delta\brho$ we have:
\begin{equation}
\label{eqn:twin-system-limits}
\delta\brho \to 0 \quad \Rightarrow \quad
\voneTauoneMinus \to \vminus, ~~
\voneTautwo \to \vplus, ~~
\voneTauonePlus \to \vplus; \quad
\vtwoTauone\to \vminus, ~~  
\vtwoTautwoMinus \to \vminus, ~~
\vtwoTautwoPlus \to \vplus.
\end{equation}
\end{definition}

Let $\bQ\afterimp,\bQ\beforeimp \in \mathds{R}^{n{\times}p}$  be the sensitivities of the generalized positions \eqref{eq:sensitivity-matrix-Q} right before and right after the event, respectively. Let $\bV\afterimp,\bV\beforeimp \in \mathds{R}^{n{\times}p}$  be the sensitivities of the generalized velocities \eqref{eq:sensitivity-matrix-V} right before and right after the event, respectively. Our methodology to find these sensitivities is to first evaluate the states $\bq_1(t),\bv_1(t)$ and $\bq_2(t),\bv_2(t)$ of each of the twin perturbed systems at both $\tau_1$ and $\tau_2$. The sensitivities are obtained from their definition by taking differences of states of the two systems, dividing them by the perturbation in parameters, and taking the limits, for example:
\[
\begin{split} 
Q_i\beforeimp = \lim_{\delta\brho_i \to 0}\, \frac{\qtwoTauone - \qoneTauone}{ \delta\brho_i},\quad
V_i\afterimp = \lim_{\delta\brho_i \to 0}\,  \frac{\vtwoTautwoPlus - \voneTautwo}{ \delta\brho_i}, 
\quad i=1,\dots,p.
\end{split} 
\]
\begin{figure} [H]
	\begin{center}
	\includegraphics[width=80mm]{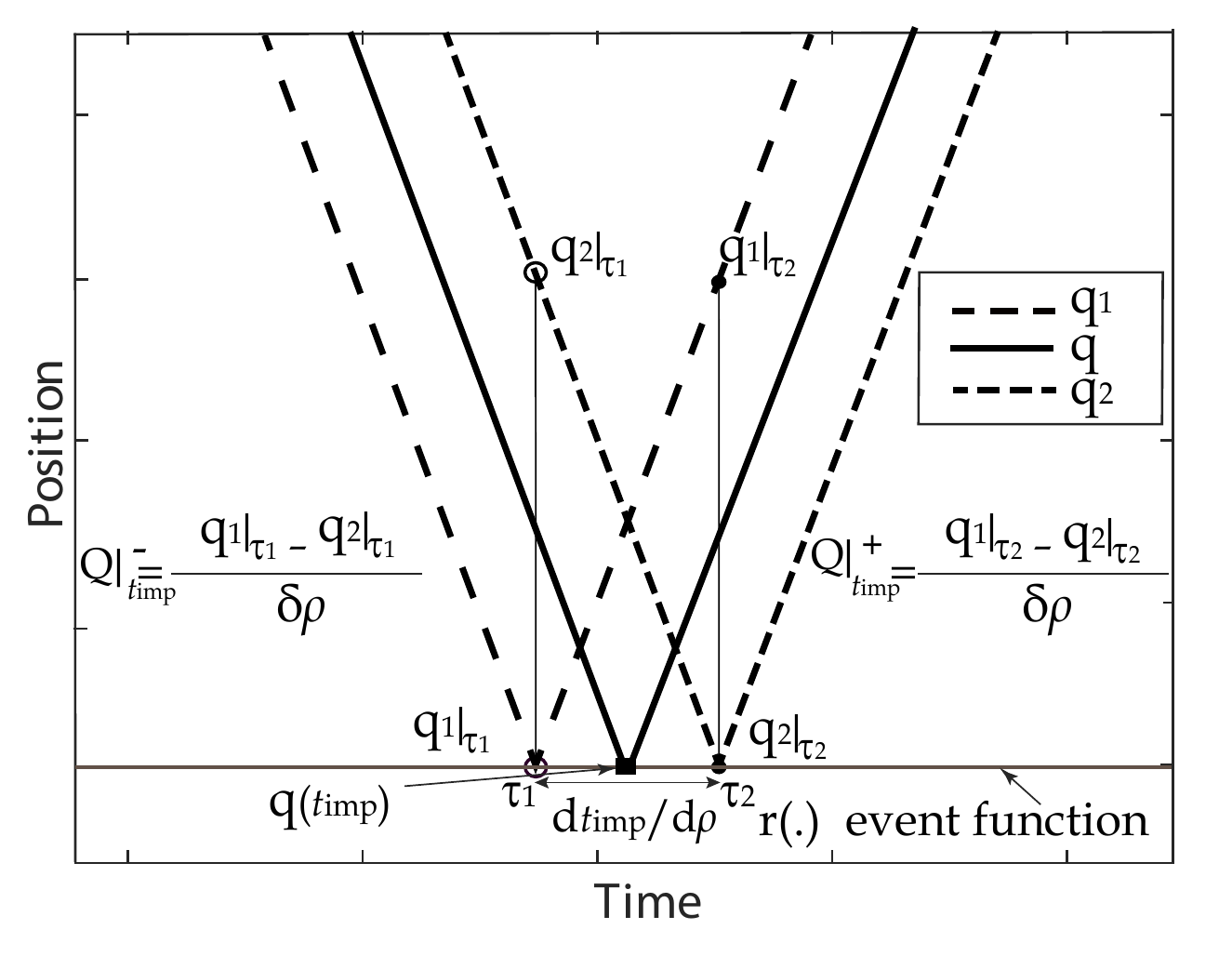}
	\end{center}
	\caption{Schematic visualization of the jump in the sensitivity of the position.}
	\label{fig:im1}
\end{figure}
\begin{figure} [H]
	\begin{center}
	\includegraphics[width=80mm]{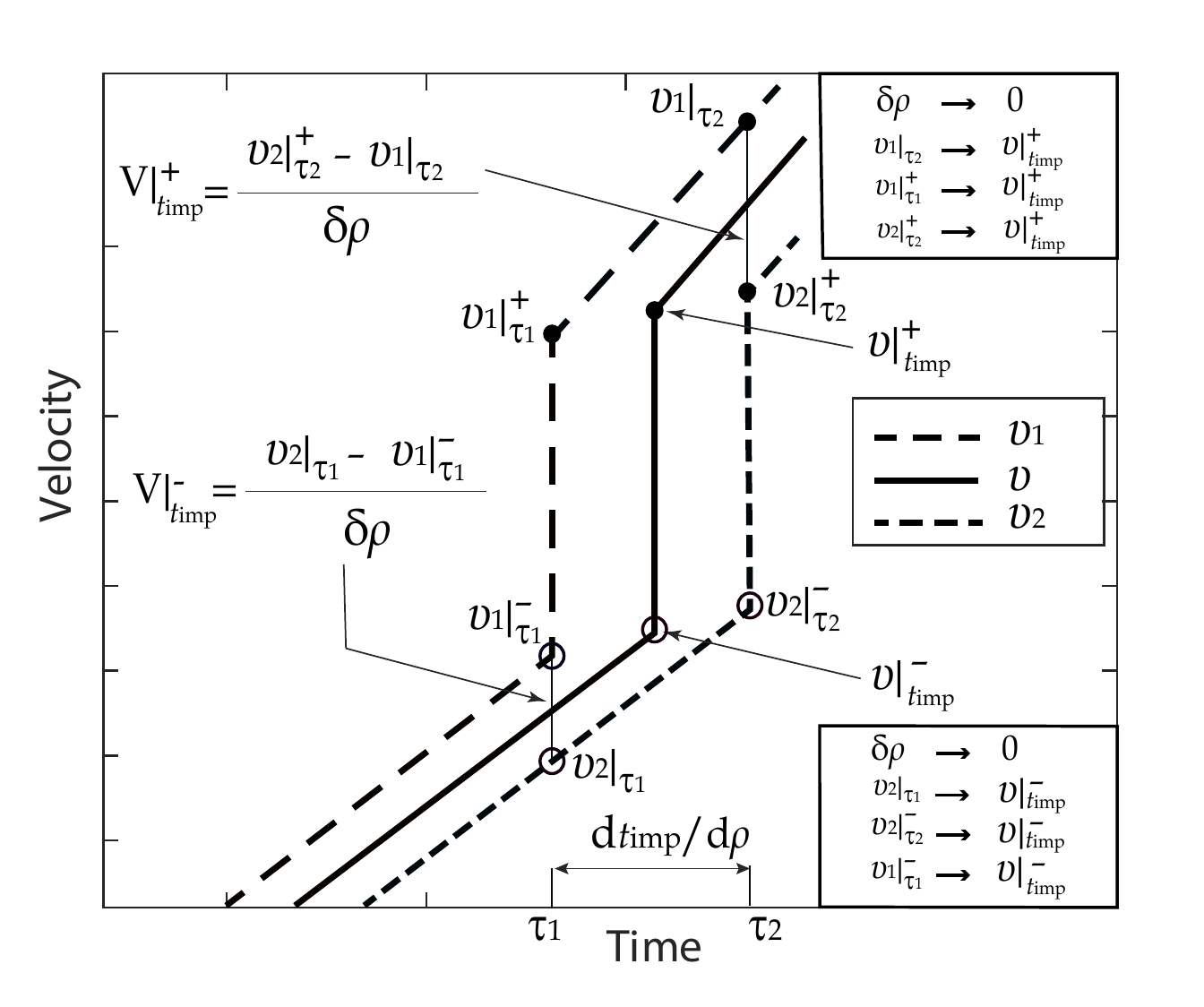}
	\end{center}
	\caption{Schematic visualization of the jump in the sensitivity of the velocity. }
	\label{fig:im2}
\end{figure}

\subsection{The sensitivity of the time of event with respect to the system parameters }
%

\begin{theorem}[Sensitivity of the time of event \cite{Donze07systematicsimulation, backer1966jump, hiskens2000trajectory,Barton1999,Barton1998,rozenvasser1967general}] 
\label{thm:sensitivity-time-of-invent}
Let $r(\cdot) \in \mathds{R}$ be the scalar event function defined by Eq.~\eqref{eq:event_function}, and $dr/d\bq \in \mathds{R}^{1 \times n }$ be its Jacobian. The sensitivity of the time of event with respect to the system parameters is:
\begin{eqnarray}
\label{eq:sensitivity-time-of-invent}
\frdtdrho=
- \, \dfrac{\displaystyle \frac{d r}{d\bq}\left(\qtimp\right)  \cdot \bQ\beforeimp}{\displaystyle \frac{d r}{d\bq}\left(\qtimp\right)  \cdot  \vminus }   
\in \mathds{R}^{1 \times p}.
\end{eqnarray}
\end{theorem}
\begin{pf}
The time at which the event function becomes zero is indirectly dependent on the system parameters $\brho$. We evaluate the derivative of equation \eqref{eq:event_function} with respect to the system parameters:
\begin{equation}
\label{eq:event_function1}
0 = r \ \bigl( \bq(\timp,\brho) \bigr) \quad \Rightarrow \quad
0=\frdrdrho=\frac{d r}{d\bq}\, \left( \frac{d\bq}{d\brho} + \dbq \, \frdtdrho \right).
\end{equation} 
Rearrange the terms in Eq.~\eqref{eq:event_function1} to obtain Eq.~\eqref{eq:sensitivity-time-of-invent}.
\qed
\end{pf}
%

\subsection{The jump in the sensitivity of the position state vector due to the event}
This section provides the jumps in the sensitivities of the position state vector $\bq(t)$ at the time of event  \cite{Donze07systematicsimulation, backer1966jump, hiskens2000trajectory,Barton1999,Barton1998,rozenvasser1967general}. Due to the nonzero inertia the position state variable is continuous in time \eqref{eq:jump-in-position}. However, its sensitivity can be discontinuous at the time of event, as established next.

\begin{theorem}[Jump in position sensitivity \cite{Donze07systematicsimulation, backer1966jump, hiskens2000trajectory, barton2002modeling}] 
\label{thm:jump-in-position-sensitivity}
Let $\vplus,\vminus \in \mathds{R}^{n}$ be the generalized velocity state vectors  after and before the invent, respectively; the corresponding velocity jump function was introduced in Eq.~\eqref{eq:jump-in-velocity}. Let $\bQ\afterimp$ and $\bQ\beforeimp \in \mathds{R}^{n{\times}p}$  be the sensitivities of the generalized position state vectors after and before the invent, respectively. The jump equation of the sensitivities of the generalized position state vector is:
\begin{eqnarray}
\label{eq:jump-position-sensitivity}
\bQ\afterimp = \bQ\beforeimp - \bigg( \vplus - \vminus \bigg) \cdot \frdtdrho.
\end{eqnarray}
\end{theorem}  
\begin{pf} 
Consider the twin perturbed systems from Definition \ref{def:perturbed-twins}. The evolution of positions is illustrated in Fig.~\ref{fig:im1}, where the two different dashed line trajectories represent the position variables of the two perturbed systems. 
The jump in the velocity state variables occurs at time $\tau_1$ only for the first system. The position variables at time $\tau_2$ for both systems are:
\begin{eqnarray}
\label{eq:jump-in-position-sensitivity-1}
\begin{split}
\qoneTautwo &= \qoneTauone + h\big( \voneTauoneMinus \big) \, \delta{\tau}, \\ 
\qtwoTautwo  &= \qtwoTauone + \vtwoTauone\, \delta{\tau}.
\end{split}
\end{eqnarray}
Subtract the two equations and scale by the perturbation in the parameters:
\begin{equation}
\label{eq:jump-in-position-sensitivity-2}
\frac{ \qtwoTautwo - \qoneTautwo }{\delta\brho}=  
-\left( \voneTauonePlus - \vtwoTauone \right)\,\frac{\delta{\tau}}{\delta\brho}  + 
\frac{\qtwoTauone - \qoneTauone}{\delta\brho}.
\end{equation}
Using \eqref{eqn:twin-system-limits} and taking the limit $\delta\brho \to 0$ in \eqref{eq:jump-in-position-sensitivity-2} we obtain \eqref{eq:jump-position-sensitivity}. The trajectory state differences are illustrated by the vertical lines in Fig.~\ref{fig:im1}.  
\end{pf}

\subsection{The jump in the sensitivity of the velocity state vector due to the event}
This section provides the jumps in the sensitivities of the velocity state vector $\bv(t)$ at the time of event  \cite{hiskens2000trajectory,Barton1999,Barton1998,rozenvasser1967general} corresponding to the jump function \eqref{eq:jump-in-velocity}.

\begin{theorem}[Jump in velocity sensitivity.]  \cite{hiskens2000trajectory,Barton1999,Barton1998,rozenvasser1967general}
\label{thm:jump-in-velocity-sensitivity}
Let $\dqrhoplus,\dqrhominus  \in \mathds{R}^{n{\times}p}$ be the sensitivities of the generalized position state vectors after and before the invent, respectively. Let $\vplus$ and $\vminus \in \mathds{R}^{n}$ be the velocity state vectors after and before the invent affected by the jump function Eq.~\eqref{eq:jump-in-velocity} , respectively.
Let $\ddbq\afterimp$ and $\ddbq\beforeimp \in \mathds{R}^{n}$ be the generalized acceleration state vectors after and before the invent, respectively.  The jump equation of the sensitivities of the generalized velocity state vector is:
\begin{eqnarray}
\label{eq:jump-in-velocity-sensitivity}
\dqrhoplus &=& h_\bq\beforeimp \cdot \qrhominus +h_\bv\beforeimp \cdot  \dqrhominus 
+\left( h_\bq\beforeimp  \cdot \vminus -\ddbq \timpplus + h_\bv\beforeimp  \cdot \ddbq \beforeimp  {+ h_t}\beforeimp \right) \cdot \frdtdrho {+ h_\brho\beforeimp},
\end{eqnarray}
where the Jacobians of the jump function are:
\[
\begin{split}
h_t\beforeimp &:= \frac{\partial h}{\partial t}\big(\timp, \qtimp, \vminus ,\brho\big) \in \Re^{f}, \qquad
h_\bq\beforeimp := \frac{\partial h}{\partial\bq}\big(\timp, \qtimp, \vminus ,\brho\big)\in \Re^{f \times n}, \\
h_\bv\beforeimp &:= \frac{\partial h}{\partial\bv}\big(\timp,\qtimp, \vminus,\brho \big)\in \Re^{f \times f}, \qquad
h_\brho\beforeimp := \frac{\partial h}{\partial\brho}\big(\timp,\qtimp, \vminus,\brho \big)\in \Re^{f \times p}.
\end{split}
\]
\end{theorem}

\begin{pf}
We consider again the twin perturbed systems from Definition \ref{def:perturbed-twins}. The jumps in velocities are illustrated in Fig.~\ref{fig:im2}. The velocities for each system are determined as follows:
\begin{equation}
\label{eq:jump sensitivity velocity}
\begin{split}
\voneTautwo &=\voneTauonePlus  + \feom \Big(\,\tau_1, \,\qoneTauone, \,\voneTauonePlus, \,\brho_1\, \Big) \, \delta{\tau},\\
 &=h\big( \, \tau_1,  \, \qoneTauone,  \,\voneTauoneMinus, \,\brho_1\, \big) + \feom \Big(\,\tau_1, \,\qoneTauone, \,h\Big(\, \tau_1, \,\qoneTauone, \,\voneTauoneMinus, \,\brho_1\, \Big), \,\brho_1\, \Big) \, \delta{\tau},\\ 
\vtwoTautwoPlus &= h \ \left(\, \tau_2,  \,\qtwoTautwo, \,\vtwoTautwoMinus,  \,\brho_2\,  \,\right) \\ 
 &= h \ \Big(\, \tau_2,  \,\qtwoTauone+\vtwoTauone \ \delta{\tau},\; \vtwoTauone  + \feom \Big(\tau_1, \,\qtwoTauone, \,\vtwoTauone, \,\brho_2 \Big) \, \delta{\tau}, \,\brho_2\,  \, \Big) \\
&\approx  h\big(\, \tau_2,  \,\qtwoTauone,\vtwoTauone, \,\brho_2\,  \,  \big) \; + \;
\frac{d h}{d \bq}\big(\qtwoTauone,\vtwoTauone \big) \cdot \vtwoTauone \ \delta{\tau} \\
&+\frac{d h}{d \bv}\big(\qtwoTauone,\vtwoTauone \big)  \cdot
 \feom \Big(\,\tau_1, \,\qtwoTauone, \,\vtwoTauone, \,\brho_2\, \Big) \, \delta{\tau}, 
\end{split}
\end{equation}
where ${\feom}$ is the instantaneous acceleration of the system from Eq.~\eqref{eq:EOM-ODE}. The last relation represents a linearization (first order Taylor expansion) that is infinitely accurate since $\delta\tau$ is infinitesimally small. 
The scaled difference between the velocity state vectors at the time of event is :
\begin{eqnarray*}
\frac{\vtwoTautwoPlus - \voneTautwo}{ \delta\brho} &\approx&
\frac{h\big(\,\tau_2, \,\qtwoTauone,\vtwoTauone, \,\brho_2\,\big) - h\big(\,\tau_1, \,\qoneTauone,\voneTauoneMinus, \,\brho_1\,\big) }{\delta\brho} \\
&& -\feom \Big(\, \tau_1,\, \qoneTauone,\, h \big(\qoneTauone, \voneTauoneMinus\big),\, \brho_1\, \Big)\, \frac{\delta{\tau}}{\delta\brho} \\ 
&& +\frac{d h}{d \bq}\big(\qtwoTauone,\vtwoTauone \big) \cdot \vtwoTauone 
\cdot \frac{\delta{\tau}}{\delta\brho}
+\frac{d h}{d \bv}\big(\qtwoTauone,\vtwoTauone\big) \cdot
 \feom \Big(\,\tau_1, \,\qtwoTauone, \,\vtwoTauone, \,\brho_2\, \Big) \cdot \frac{\delta{\tau}}{\delta\brho} \\
&\approx&
\frac{d h}{d t}\big(\qoneTauone,\voneTauoneMinus \big) \cdot \frac{\tau_2 - \tau_1}{ \delta\brho} +
\frac{d h}{d \bq}\big(\qoneTauone,\voneTauoneMinus \big) \cdot \frac{\qtwoTauone - \qoneTauone}{ \delta\brho} +
\frac{d h}{d \bv}\big(\qoneTauone,\voneTauoneMinus \big) \cdot \frac{\vtwoTauone - \voneTauoneMinus}{ \delta\brho} \\
&& + \frac{d h}{d \brho}\big(\qoneTauone,\voneTauoneMinus \big)
-\feom \Big(\tau_1,\,\qoneTauone,~ h \big(\qoneTauone,\voneTauoneMinus),\,\brho_1 \Big)\, \frac{\delta{\tau}}{\delta\brho} \\ 
&& +\frac{d h}{d \bq}\big(\qtwoTauone,\vtwoTauone \big) \cdot \vtwoTauone
\cdot \frac{\delta{\tau}}{\delta\brho}
 +\frac{d h}{d \bv}\big(\qtwoTauone,\vtwoTauone \big) \cdot
 \feom \Big(\tau_1, \,\qtwoTauone, \,\vtwoTauone, \,\brho_2 \Big) \cdot  \frac{\delta{\tau}}{\delta\brho}. 
\end{eqnarray*} 
Taking the limit $\delta\brho \to 0$ and using Eq. \eqref{eqn:twin-system-limits} yields:
\[
\begin{split}
&\frac{d h}{d \bq}\big(\qoneTauone,\voneTauoneMinus \big) \to
h_\bq\beforeimp, \quad
\frac{d h}{d \bq}\big(\qtwoTauone,\vtwoTauone \big) \to
h_\bq\beforeimp, \\
&\frac{d h}{d \bv}\big(\qoneTauone,\voneTauoneMinus \big) \to
h_\bv\beforeimp, \quad
\frac{d h}{d \bv}\big(\qtwoTauone,\vtwoTauone \big) \to
h_\bv\beforeimp,\\
&\frac{d h}{d t}\big(\qoneTauone,\voneTauoneMinus \big) \to
h_t\beforeimp, \quad
\frac{d h}{d \brho}\big(\qoneTauone,\voneTauoneMinus \big) \to
h_\brho\beforeimp,\\
&\feom \Big(\tau_1, \,\qtwoTauone, \,\vtwoTauone, \,\brho_2 \Big)
\to \feom \Big(\timp, \,\qminus, \,\vminus, \,\brho \Big)
= \ddbq \beforeimp,\\
&\feom \Big(\tau_1, \,\qoneTauone,\voneTauonePlus, \,\brho_1 \Big)
 \to \feom \Big(\timp, \,\qplus, \,\vplus, \,\brho \Big)
= \ddbq \timpplus.
\end{split} 
\]
which leads to Eq.~\eqref{eq:jump-in-velocity-sensitivity}.

For simplicity we denote the derivatives of the jump function  with respect to $\zeta \in \{ t, \bq,\bv,\brho \}$ by:
\[
\begin{split}
&\frac{d h}{d \zeta}\big(\tau_1, \,\qoneTauone, \,\voneTauoneMinus, \,\brho_1\, \big) =
\frac{d h}{d \zeta}\big(\qoneTauone,\voneTauoneMinus \big) \\
&\frac{d h}{d \bq}\big(\tau_1, \,\qtwoTauone, \,\vtwoTauone, \,\brho_2\, \big) =
\frac{d h}{d \zeta}\big(\qtwoTauone,\vtwoTauone \big) \\
\end{split} 
\]
\qed
\end{pf} 
%


\begin{theorem}[Events that only change the acceleration]
\label{thm:ODE-change-in-acceleration}
We now consider an event where the system undergoes a sudden change of the equation of motions \eqref{eq:EOM} at $\timp$:
\[
\ddbq \beforeimp = \feom^- \big(\timp,\qtimp,\vtimp, \brho\big) =:  \feom^-\atimp
\quad \stackrel{\rm event}{\longrightarrow} \quad
\ddbq \timpplus = \feom^+ \big(\timp,\qtimp,\vtimp, \brho\big)=:  \feom^+\atimp.
\]
There is no abrupt jump in the system velocity, $\vplus=\vminus$, and therefore the jump function \eqref{eq:jump-in-velocity} is identity. Let $\dqrhoplus,\dqrhominus  \in \mathds{R}^{n{\times}p}$ be the sensitivities of the generalized position state vectors right after and before the event, respectively.  Let $\ddbq\afterimp$ and $\ddbq\beforeimp \in \mathds{R}^{n}$ be the generalized acceleration state vectors right after and before the event, respectively.  
The jump equation of the sensitivities of the generalized velocity state vector is:
\begin{eqnarray}
\label{eq:jump-in-velocity-sensitivity-changeacc}
\dqrhoplus = \dqrhominus -\left( \ddbq \timpplus - \ddbq \beforeimp \right) \cdot \frdtdrho
= \dqrhominus -\left( \feom^+\atimp - \feom^-\atimp \right) \cdot \frdtdrho.
\end{eqnarray}
\end{theorem}

\begin{pf}
For the type of events under consideration we have that:
\[
\frac{d h}{d \bq}= \bzero, \quad \frac{d h}{d \bv}=\bI.
\]
Using this in Eq. \eqref{eq:jump-in-velocity-sensitivity} leads to Eq.~ \eqref{eq:jump-in-velocity-sensitivity-changeacc}. 

\qed
\end{pf}

\subsection{The jump in the sensitivity of the cost functional due to the event}
We now consider the sensitivity of the quadrature variable $\bz(t)$. Due to the integral form of Eq. \eqref{eqn:quadrature-variable} defining $\bz$,  the quadrature variable is continuous in time: 
\[
\zplus= \zminus = z\atimp.
\]
However, its sensitivity can be discontinuous at the event time, as established next.

\begin{theorem}[Jump in quadrature sensitivity.] 
\label{thm:jump-in-quadrature-sensitivity}
Let $\Zplus$ and $\Zminus$, with  $\bZ \in \mathds{R}^{p}$, be the sensitivities of the quadrature variable $\bz(t)$ (Definition \ref{The quadrature sensitivity}) right after and right before the event, respectively.
Let 
\[
\gplus := \gplusarg,
\quad
\gminus := \gminusarg,
\]
be the running cost function evaluated right after and right before the event, respectively. 
The sensitivity of the cost functional changes during the event as follows:
\begin{eqnarray}
\label{eq:jump-in-quadrature-sensitivity}
\Zplus= \Zminus - \Big( \gplus - \gminus \Big) \cdot \frdtdrho.
\end{eqnarray}
\end{theorem} 
\begin{pf}
Consider again the twin perturbed systems from Definition \ref{def:perturbed-twins}, and evaluate the associated quadrature variables \eqref{eqn:quadrature-variable} at the event:
\begin{equation}
\begin{split}
\zoneTautwo &=
\zoneTauone  
+ \int_{\tau_1}^{\tau_2}  \tilde{g}\bigl(t,\bq_1(t),\bv_1(t),\brho_1\bigr) \, dt
= \zoneTauone   
+ \tilde{g}\bigl(\,\tau_1,\,\bq_1|_{\tau_1},\,\bv_1|_{\tau_1}^+,\,\brho_1\,\bigr) \, \delta{\tau}, \\
\ztwoTautwo  &=
\ztwoTauone 
+ \int_{\tau_1}^{\tau_2}  \tilde{g}\bigl(t,\bq_2(t),\bv_2(t),\brho_2\bigr)\, dt
= \ztwoTauone  
+ \tilde{g}\bigl(\,\tau_2,\,\bq_2|_{\tau_2},\,\bv_2|_{\tau_2}^-,\,\brho_2\,\bigr) \, \delta{\tau}.
\end{split}
\end{equation}
Subtract the two equations and scale by the parameter perturbation to obtain:
\[
\begin{split}
\frac{ \ztwoTautwo - \zoneTautwo }{\delta\brho}  &=  
\frac{ \ztwoTauone- \zoneTauone }{\delta\brho} 
 + \left(
\tilde{g}\bigl(\,\tau_2,\,\bq_2|_{\tau_2},\,\bv_2|_{\tau_2}^-,\,\brho_2\,\bigr)
-
\tilde{g}\bigl(\,\tau_1,\,\bq_1|_{\tau_1},\,\bv_1|_{\tau_1}^+,\,\brho_1\,\bigr)
\right)\, \frac{\delta\tau}{\delta\brho}.
\end{split}
\]
Taking the limit $\delta\brho \to 0$ leads to Eq.~\eqref{eq:jump-in-quadrature-sensitivity}. 
\qed
\end{pf}
%
%
\begin{remark}
When there are multiple events along the trajectory jumps in sensitivity ~\eqref{eq:jump-in-quadrature-sensitivity} will happen for each one.  The jump of the quadrature variable $\bZ$ is governed by the values of the cost function $g$ before and after the event. 
\end{remark}

\section{Direct sensitivity analysis for constrained multibody systems with smooth trajectories}
\label{sec:multibody-smooth}

This section reviews the direct sensitivity analysis for constrained systems governed by differential algebraic equations (DAEs). The presentation follows the authors' earlier work
\cite{Sandu_2013_sensitivity_ODE_multibody,Sandu_2014_sensitivity_ODE_multibody,zhu2014mbsvt,Zhu_2014_PhD,Sandu2015dynamic,Sandu_2017_vehicle-optimization}.

\subsection{Representation of constrained multibody systems}
Constrained multibody systems must satisfy the following kinematic constraints: 
\begin{subequations}
\label{eq:ConstraintsEq}
\begin{eqnarray}
\label{eq:ConstraintsEq-position}
\bzero  &=& \bPhi, \\
\label{eq:ConstraintsEq-velocity}
\bzero &=& \dPhi =\dPhidq \, \dbq + \bPhi_t \quad \Rightarrow \quad  \dPhidq \bv = -\bPhi_t, \\
\label{eq:ConstraintsEq-acceleration}
\bzero &=& \ddPhi =\dPhidq \, \ddbq + \dPhidqq \, ( \dbq, \dbq ) + \dPhidtdq \, \dbq + \bPhi_{t, \, t} \quad \Rightarrow \quad  \dPhidq \, \dbv = -  ( \dPhidq  \, \bv)\, \bv  - \dPhidtdq \, \bv - \bPhi_{t, \, t} := \Faccel.
\end{eqnarray}
\end{subequations}
Here \eqref{eq:ConstraintsEq-position} is a holonomic position constraint equation $\bPhi(t,\bq,\brho)=\bzero$, where $\bPhi : \mathds{R}^{1+n+p} \to \mathds{R}^{m}$ is a smooth `position constraint' function. The velocity \eqref{eq:ConstraintsEq-velocity} and the acceleration \eqref{eq:ConstraintsEq-acceleration} kinematic constraints are found by differentiating the position constraint with respect to time.
 
There are two main approaches to solve such systems, the DAE approach through direct inclusion of the algebraic constraints in the dynamics, and the ODE approach through either following locally the independent coordinates (Maggi) or through a penalty formulation.
%
%

\subsection{Direct sensitivity analysis for smooth systems in the index-3 differential-algebraic formulation}

\begin{definition}[Constrained multibody dynamics: the index-3 DAE formulation]
A constrained rigid multibody dynamics system is described by the following index-3 differential-algebraic equations (DAEs) \cite{Sandu_2014_sensitivity_ODE_multibody}:
\begin{equation}
\label{eq:EOM-DAE-index3}
\begin{cases}
\dbq &= \bv, \\
{\Mass}\left(t,\bq,\brho\right) \cdot \dbv &={\Force} \left(t,\bq,\bv,\brho\right) + \dPhidq^{\rm T}\left(t,\bq,\brho\right)\cdot\lambda, \\
\bPhi\left(t,\bq,\brho\right) &= \bzero,
\end{cases}
\quad t_0 \leq t \leq  t_F,  \quad \bq(t_0)=\bq_0(\brho),\quad\bv(t_0)=\bv_0(\brho).
\end{equation}
Unlike the ODE formulation \eqref{eq:EOM-ODE} the position vector of the system \eqref{eq:EOM-DAE-index3} is constrained by the equation  \eqref{eq:ConstraintsEq-position}. The joint forces $\bPhi_\bq^{\rm T}\,\lambda$ ensure that the system solution obeys the constraints at all points along the trajectory, and $\lambda \in \mathds{R}^{m}$ are Lagrange multipliers  associated with the position constraint \eqref{eq:ConstraintsEq-position}.
\end{definition}

%

Sensitivities of the position and velocity state variables are defined in \eqref{eq:sensitivity-of-solutions}. In addition, we need to consider the sensitivity of the Lagrange multipliers with respect to system parameters:
\begin{equation}
\label{eq:sensitivity-of-multipliers}
\Lambda(t,\brho) := \Drho\lambda(t) := \frac{d \lambda}{d \brho}(t, \brho) \in \mathds{R}^{m \times p}.
\end{equation}

\begin{definition}[TLM of the index-3 DAE formulation]
Sensitivities of solutions \eqref{eq:sensitivity-of-solutions} and multipliers \eqref{eq:sensitivity-of-multipliers} of the system \eqref{eq:EOM-DAE-index3} with respect to parameters evolve according to the tangent linear model derived in \cite{Sandu_2013_sensitivity_ODE_multibody,Sandu_2014_sensitivity_ODE_multibody,zhu2014mbsvt,Zhu_2014_PhD,Sandu2015dynamic,Sandu_2017_vehicle-optimization}:
\begin{equation*}
\begin{cases}
&\dot{\bQ} = \bV, \\
&\Mass\cdot \dot{\bV} =
\Force_\bv\cdot \bV
-\left(\Mass_{\bq}\, {\dot{\bv}} + \bPhi_{\bq,\bq}^{\rm T}\, \lambda -\Force_\bq \right) \cdot \bQ 
- \dPhidq^{\rm T}\cdot \Lambda + \Force_{\brho}-\Mass_{\brho}\, {\dot{\bv}} 
- \bPhi_{\bq,\brho}^{\rm T} \lambda, \\
&\dPhidq \cdot \bQ = -\bPhi_{\brho},
\end{cases}
\end{equation*}
with initial conditions given by Eq. \eqref{eq:TLM-IC}.
\end{definition}

\subsection{Direct sensitivity analysis for smooth systems in the index-1 differential-algebraic formulation}

\begin{definition}[Constrained multibody dynamics: the index-1 DAE formulation]
The index-1 formulation of the equations of motion is obtained by replacing the position constraint \eqref{eq:ConstraintsEq-position} in \eqref{eq:EOM-DAE-index3} with the acceleration constraint \eqref{eq:ConstraintsEq-acceleration}:
\begin{equation}
\label{eq:EOM-DAE-index1}
\begin{bmatrix}
\bI & \bzero & \bzero \\
\bzero & {\Mass}\left(t,\bq,\brho\right) & \dPhidq^{\rm T}\left(t,\bq,\brho\right) \\
\bzero & \dPhidq\left(t,\bq,\brho\right) & \bzero
\end{bmatrix}
\cdot
\begin{bmatrix}
\dbq \\ \dbv \\ \lambda
\end{bmatrix}
=
\begin{bmatrix}
\bv \\
{\Force} \left(t,\bq,\bv,\brho\right)  \\
\Faccel\left(t,\bq,\bv,\brho\right)
\end{bmatrix},
\quad t_0 \leq t \leq  t_F,  \quad \bq(t_0)=\bq_0(\brho),\quad\bv(t_0)=\bv_0(\brho),
\end{equation}
or equivalently,
\begin{equation}
\label{eq:DAE-index1}
\dbq = \bv, \qquad
\begin{bmatrix}
\dbv \\ \lambda
\end{bmatrix}
=
\begin{bmatrix}
{\Mass} & \dPhidq^{\rm T} \\
\dPhidq & \bzero
\end{bmatrix}^{-1}
\cdot
\begin{bmatrix}
{\Force}   \\
\Faccel
\end{bmatrix}=
\begin{bmatrix}
\fdaedv   \\
\fdaelb
\end{bmatrix}=
\fdae \fin.
\end{equation}
The algebraic equation has the form $\fdaelb -\lambda=0$.
\end{definition}

\begin{definition}[TLM of the index-1 DAE formulation]
Sensitivities of solutions \eqref{eq:sensitivity-of-solutions} and multipliers \eqref{eq:sensitivity-of-multipliers} of the system  \eqref{eq:EOM-DAE-index1} with respect to parameters evolve according to the tangent linear model derived in \cite{Sandu_2013_sensitivity_ODE_multibody,Sandu_2014_sensitivity_ODE_multibody,zhu2014mbsvt,Zhu_2014_PhD,Sandu2015dynamic,Sandu_2017_vehicle-optimization}:
\begin{equation*}
\label{eq:TLM-DAE-index1}
\begin{bmatrix}
\bI & \bzero & \bzero \\
\bzero & {\Mass} & \dPhidq^{\rm T} \\
\bzero & \dPhidq & \bzero
\end{bmatrix}
\cdot
\begin{bmatrix}
\dot{\bQ} \\ \dot{\bV} \\ \Lambda
\end{bmatrix}
=
\begin{bmatrix}
\bV \\
\Force_\bv\cdot  \bV
-\left(\Mass_{\bq} \, \dbv + \dPhidqq^{\rm T}\, \lambda  -\Force_\bq \right) \cdot \bQ  
+ \Force_{\brho}-\Mass_{\brho} \,\dot{\bv} 
- \bPhi_{\bq, \, \brho}^{\rm T}\, \lambda\\
\Faccel_\bv \cdot \bV
-\left(\dPhidqq \, \dbv  -\Faccel_{\bq} \right) \cdot \bQ +
\Faccel_{\brho} - \dPhidqdrho \, \dbv
\end{bmatrix},
\end{equation*}
with initial conditions given by Eq. \eqref{eq:TLM-IC}.
\end{definition}

\begin{definition}[Cost function] 
\label{def:CostFunction-DAE}
Following Definition \ref{def:CostFunction-DAE}, consider a smooth scalar ``trajectory cost function'' $\bg$ and a smooth scalar ``terminal cost function'' $w$. A general cost function has the form:
\begin{eqnarray*}
\label{eq:CostFunction-DAE}
&&\psi=  \int_{t_0}^{t_F} {\gdae \  dt} + \w.
\end{eqnarray*}
\end{definition}
Note that the trajectory cost function \eqref{eq:CostFunction} depends on both accelerations $\dbv$ and on the Lagrange multipliers $\lambda$. These are not independent variables and they can be resolved in terms of positions and velocities using \eqref{eq:EOM-DAE-index1}, to obtain an equivalent regular trajectory cost function:
\[
\bg\bigl(\,t,\,\bq,\,\bv,\,\dbv(t,\bq,\bv,\brho),\,\lambda(t,\bq,\bv,\brho),\,\brho\,\bigr) = \tilde{g}\bigl(\,t,\,\bq,\,\bv,\,\brho\,\bigr).
\] 
We keep accelerations and Lagrange multipliers (constraint forces) as explicit parameters in the cost function \eqref{eq:CostFunction-DAE} in order to give additional flexibility in practical applications. In addition, we define the `quadrature' variable $z(t) \in \mathds{R}$ as follows:
\begin{subequations} 
\label{eqn:DAE-quadrature-variable}
\begin{eqnarray}
\label{eqn:DAE-quadrature-variable-int}
\bz(t,\brho) &:=&  \int_{t_0}^{t} { \tilde{g}\bigl(t,\bq(t,\brho),\bv(t,\brho),\brho\bigr) \,  dt} \quad \Leftrightarrow \quad \\
\label{eqn:DAE-quadrature-variable-dif}
\dbz(t,\brho) &=&  \bg\bigl(\,t,\,\bq,\,\bv,\,\dbv,\,\lambda,\,\brho\,\bigr)  = \tilde{g}\bigl(t,\bq(t,\brho),\bv(t,\brho), \brho\bigr),  \quad t_0 \leq t \leq  t_F, \quad z(t_0,\brho)=0.
\end{eqnarray}
\end{subequations}

\begin{definition}[The DAE quadrature sensitivity] 
Similarly, let the `quadrature sensitivity' vector $Z(t,\brho)$ be the Jacobian of the `quadrature' variable $z(t,\brho)$ Eq. ~\eqref{eqn:DAE-quadrature-variable-int} with respect to the parameters $\brho$:
\[
Z_i(t,\brho) := \frac{\partial z(t,\brho)}{\partial \brho_i}, ~~ i=1,\dots,p; \quad
Z(t,\brho) := \nabla_\brho z(t,\brho) = \begin{bmatrix} Z_1(t,\brho) \cdots Z_p(t,\brho) \end{bmatrix} \in \mathds{R}^{1 \times p}.
\]
The time evolution equations of the quadrature sensitivities are given by the TLM obtained by differentiating 
\eqref{eqn:DAE-quadrature-variable-dif} with respect to the parameters:
\begin{equation}
\label{eq:TLM-quadrature}
\begin{split}
\dbZ_i & =  \frac{d\, \bg\bigl(\,t,\,\bq,\,\bv,\,\dbv,\,\lambda,\,\brho\,\bigr)}{d\,\brho_i} \\
&= g_\bq\cdot \bQ_{i} +  g_\bv\cdot \bV_{i}  + g_{\dbv}\cdot \frac{d \fdaedv}{d \brho}
+g_{\lambda}\cdot \frac{d \fdaelb}{d \brho} + g_{\brho_i} \\
&= \big(g_\bq +  g_{\dbv} \cdot \fdaedvq +g_{\lambda} \cdot \fdaelbq \big)\cdot \bQ_{i}  
+  \big(g_\bv +  g_{\dbv} \cdot \fdaedvv +g_{\lambda} \cdot \fdaelbv \big)\cdot \bV_{i}
+   g_{\dbv}\cdot\fdaelbq +g_{\lambda}\cdot\fdaelbv+g_{\brho_i} , \\
& i=1,\dots,p, \quad t_0 \le t \le t_F, \quad  \bZ_i(t_0,\brho) = 0.
\end{split}
\end{equation}
\end{definition}

\begin{definition} [Canonical index-1 sensitivity DAE]
\label{def:canonical-dae-sensitivity}
The canonical DAE system for the solution given by \eqref{eq:DAE-index1}, the DAE TLM given by \eqref{eq:TLM-DAE-index1}, and the sensitivity quadrature equations given by \eqref{eq:TLM-quadrature} need to be solved together forward in time, leading to the canonical sensitivity DAE that computes the derivatives of the cost function with respect to the system parameters $\brho$ for smooth systems:
\begin{eqnarray}
\label{eq:canonical-DAE-sensitivity}
\begin{bmatrix}
\dbq \\ \dbv \\ \lambda \\ \dbz \\ 
\big[ \dot{\bQ}_i  \big]_{i=1,\dots,p} \\
\big[ \dbV_i  \big]_{i=1,\dots,p} \\ 
\big[ \Lambda_i  \big]_{i=1,\dots,p} \\ 
\big[ \dbZ_i  \big]_{i=1,\dots,p}
\end{bmatrix}=
\begin{bmatrix}
\bv \\
\fdaedv\\
\fdaelb\\
 \tilde{\bg} \\
\left[  \bV_i  \right]_{i=1,\dots,p}  \\
\left[ \fdaedvq \bQ_i +\fdaedvv \bV_i +\fdaedvrhoi  \right]_{i=1,\dots,p} \\
\left[ \fdaelbq \bQ_i +\fdaelbv \bV_i +\fdaelbrhoi  \right]_{i=1,\dots,p} \\
\left[  \big(g_\bq +  g_{\dbv} \, \fdaedvq +g_{\lambda} \, \fdaelbq \big)\cdot \bQ_{i}  
+  \big(g_\bv +  g_{\dbv} \, \fdaedvv +g_{\lambda} \, \fdaelbv \big)\cdot \bV_{i}
+   g_{\dbv}\,\fdaelbq +g_{\lambda}\,\fdaelbv +g_{\brho_i}  \right]_{i=1,\dots,p}
\end{bmatrix},
\end{eqnarray}
where the state vector of the canonical index-1 sensitivity DAE \eqref{eq:canonical-DAE-sensitivity} is :
\begin{eqnarray}
\label{eq:canonical-DAEsensitivity-state}
\bX = \left[   \, \bq^{\rm T}, \, \bv^{\rm T}, \lambda^{\rm T}, \, \bz, ~
\bQ_1^{\rm T},  \dots , \bQ_p^{\rm T}, \,
\bV_1^{\rm T},  \dots , \bV_p^{\rm T}, \,
\Lambda_1^{\rm T},  \dots , \Lambda_p^{\rm T}, \,
\bZ_1, \dots, \bZ_p
\right] ^{\rm T}  \in \mathds{R}^{(n+1)(p+1)},
\end{eqnarray}
and the derivatives of the DAE function are:
\[
\fdaeq=  
\begin{bmatrix}
{\Mass} & \dPhidq^{\rm T} \\
 \dPhidq & \bzero
\end{bmatrix}^{-1}
\begin{bmatrix}
\Force_\bq -\Mass_{\bq} \, \dbv - \dPhidqq ^{\rm T}\, \lambda 
\\
\Faccel_{\bq}- \dPhidqq \dot{\bv}   
\end{bmatrix}
, \;
\fdaev= \begin{bmatrix}
{\Mass} & \dPhidq^{\rm T} \\
 \dPhidq & \bzero
\end{bmatrix}^{-1}
\begin{bmatrix}
\Force_\bv \\
\Faccel_\bv 
\end{bmatrix}
,\;
\fdaerho=\begin{bmatrix}
{\Mass} & \dPhidq^{\rm T} \\
 \dPhidq & \bzero
\end{bmatrix}^{-1}
\begin{bmatrix}
\Force_{\brho}-\Mass_{\brho} \, \dbv 
- \dPhidqdrho^{T}\, \lambda\\
\Faccel_{\brho} - \dPhidqdrho \, \dbv
\end{bmatrix}.
\]
\end{definition}

\subsection{Direct sensitivity analysis for smooth systems in the penalty ODE formulation}
\begin{definition}[Constrained multibody dynamics: the penalty ODE formulation]
Define the extended mass matrix $\overline{\Mass}:\mathds{R} \times \mathds{R}^{n} \times \mathds{R}^{n} \times \mathds{R}^{p} \rightarrow \mathds{R}^{n \times n}$ as: 
\begin{subequations}
\label{eq:EOM-ODE-penalty}
\begin{equation}
\label{eq:EOM-ODE-penalty-mass}
\overline{\Mass}\left(t,\bq,\bv,\brho\right) := \Mass\left(t,\bq,\bv,\brho\right) +\dPhidq^{\rm T}\left(t,\bq,\bv,\brho\right)\cdot \alpha\cdot \dPhidq\left(t,\bq,\bv,\brho\right), 
\end{equation}
where $\alpha \in \Re^{m \times m}$ is the penalty factor of the ODE penalty formulation.
Define the extended right hand side function $\overline{\Force}:\mathds{R} \times \mathds{R}^{n} \times \mathds{R}^{n} \times \mathds{R}^{p} \rightarrow \mathds{R}^{n}$ as:
\begin{equation}
\label{eq:EOM-ODE-penalty-force}
\overline{\Force}\left(t,\bq,\bv,\brho\right) := {\Force}\left(t,\bq,\bv,\brho\right)
-\dPhidq^{\rm T}\cdot \alpha\cdot \left(\dtdPhidq\, \bv +
\dPhi_{t}+2 \, \xi\, \omega\, \dPhi+\omega^2 {\bPhi}\right),
\end{equation}
where $\xi\in \Re$ and $\omega\in \Re$ are the natural frequency and damping ratio coefficients of the formulation, respectively, and  $\dPhi$ is the total time derivative of the kinematic constraints. The algebraic position constraints \eqref{eq:ConstraintsEq-position} are removed and an auxiliary spring-damper force is added in \eqref{eq:EOM-ODE-penalty-force} to prevent the system from deviating away from the constraints.

In the penalty formulation the EOM of a constrained rigid multibody system is the second order ODE:
\begin{equation}
\begin{cases}
\dbq &= \bv, \\
\dbv &= {\feom} \fin = \overline{\Mass}^{-1}\fin \cdot
\overline{\Force}\fin.
\end{cases}
\end{equation}
The Lagrange multipliers associated to the constraint forces are estimated as follows:
\begin{equation}
\label{eq:lambda}
\blambda^{*}=\alpha\, \left(\, \ddot{\bPhi}+
2 \, \xi\, \omega\, \dPhi+\omega^2 \,\bPhi \, \right)\,.
\end{equation}
\end{subequations}
\end{definition}
The cost function \eqref{eq:CostFunction-DAE} is formulated using the Lagrange multiplier estimates \eqref{eq:lambda}, i.e., using the trajectory cost function $\bg\left(\,t,\,\bq,\,\bv,\,\dbv,\,\lambda^{*},\,\brho\,\right)$.
Sensitivities \eqref{eq:sensitivity-of-solutions} of the position and velocity state variables of the system \eqref{eq:EOM-ODE-penalty} with respect to parameters evolve according to the tangent linear model derived in \cite{Sandu_2013_sensitivity_ODE_multibody,Sandu_2014_sensitivity_ODE_multibody,zhu2014mbsvt,Zhu_2014_PhD,Sandu2015dynamic,Sandu_2017_vehicle-optimization}:
\begin{equation}
\label{eq:TLM-ODE-penalty}
\begin{cases}
\dot{\bQ} &= \bV, \\
\overline{\Mass}\cdot \dot{\bV} &= 
\left( \overline{\Force}_\bq  - \overline{\Mass}_\bq \, \dbv\right)\cdot \bQ  
+\overline{\Force}_\bv\cdot  \bV 
+ \overline{\Force}_\brho -\overline{\Mass}_{\brho} \, \dbv,
\end{cases}
\qquad t_0 \le t \le t_F,
\end{equation}
with initial conditions given by Eq.~\eqref{eq:TLM-IC}. The derivatives 
$\overline{\Force}_\bq, \overline{\Force}_\bv, \overline{\Force}_\brho, \overline{\Mass}_{\brho},$ and $ \overline{\Mass}_\bq$ are given in  \ref{sec:AppendixA}.

\begin{definition}[Canonical ODE sensitivity system]
The canonical sensitivity ODE that computes the derivatives of the cost function with respect to the system parameters ρ for the smooth ODE penalty system Eq.~\eqref{eq:EOM-ODE-penalty} is the same than the ODE canonical system presented in \eqref{eq:canonical-ode-sensitivity} and extended to the cost function \eqref{eq:CostFunction-DAE} formulated using the Lagrange multiplier estimates.
\begin{eqnarray}
\label{eq:canonical-ode-penalty-sensitivity}
\begin{bmatrix}
\dbq \\ \dbv \\ \dbz \\ \big[ \dot{\bQ}_i \big]_{i=1,\dots,p}\\
\big[ \dot{\bV}_i \big]_{i=1,\dots,p}\\ 
\big[ \dbZ_i \big]_{i=1,\dots,p}
\end{bmatrix} =
\begin{bmatrix}
\bv \\
\feom\\
 \tilde{\bg}\\
\big[  \bV_i \big]_{i=1,\dots,p}\\
\big[ \feomq \cdot \bQ_i +
\feomv \cdot \bV_i +
\feomrhoi \big]_{i=1,\dots,p} \\
\big[ \big(g_\bq +  g_\dbv \cdot \feomq + g_{\blambda^{*}} \cdot \blambda_{\bq}^{*} \big)\cdot \bQ_{i}  
+  \big(g_\bv +  g_\dbv \cdot \feomv +g_{\blambda^{*}} \cdot \blambda_{\bv}^{*} \big)\cdot \bV_{i}+ 
\big(g_{\brho_i} +  g_\dbv\cdot\feomrhoi +  g_\dbv\cdot\blambda_{\brho}^{*}\big) \big]_{i=1,\dots,p}
\end{bmatrix},
\end{eqnarray}
where
\[
\feomq=\overline{\Mass}^{\rm -1}\left(\overline{\Force}_\bq-
\overline{\Mass}_{\bq} \dbv\right), \quad
\feomv=\overline{\Mass}^{\rm -1}\overline{\Force}_\bv,\quad
\feomrhoi=\overline{\Mass}^{\rm -1}\left(\overline{\Force}_\brho - \overline{\Mass}_\brho \dbv \right),
\]
and with the initial conditions given by Eq. \eqref{eq:TLM-IC}. The derivatives $\blambda_{\bq}^{*}, \blambda_{\bv}^{*}$ and  $ \blambda_{\brho}^{*}$ are given in \ref{sec:AppendixA}.
\end{definition}
\begin{remark} The sensitivity of the estimated Lagrange multipliers 
\begin{equation}
\label{eq:sensitivity-of-estimated-multipliers}
\Lambda^*(t,\brho) := \Drho\lambda^*(t) := \frac{d \lambda^*}{d \brho}(t, \brho) \in \mathds{R}^{m \times p}
\end{equation}
is calculated as:
		\begin{equation}
		\Lambda_i^{*}  =  \blambda_{\bq}^{*}  \; \bQ_i +\blambda_{\bv}^{*}\; \bV_i +\blambda_{\brho_i}^{*}, \quad i=1,\dots,p.
		\end{equation}
	\end{remark}

\section{Direct sensitivity analysis for hybrid constrained multibody systems}
\label{sec:multibody-constrained}

We now discuss constrained multibody systems when the dynamics is piecewise continuous in time.

\subsection{Coordinates partitioning for hybrid multibody systems}

The direct sensitivity analysis for a constrained rigid hybrid multibody dynamic system requires to find the jump conditions at the time of event. For this we need to distinguish between dependent and independent state variables and their sensitivities.

Assume that the Jacobian of the position constraint \eqref{eq:ConstraintsEq-position} has full row rank at a given configuration, $\operatorname{rank}(\bPhi_\bq)=m$. One can rearrange the columns and split the Jacobian in two submatrices:
\begin{equation}
\label{eqn:split-constraint-jacobian}
\bPhi_\bq \cdot \Permutation^{\rm T} = \big[ \bPhi_{\bq_{\rm dep}} ~~ \bPhi_{\bq_{\rm dof}} \big], \quad
 \bPhi_{\bq_{\rm dep}} \in \mathds{R}^{m \times m}, \quad
 \bPhi_{\bq_{\rm dof}} \in \mathds{R}^{m \times f}, \quad f = n-m,
\end{equation}
such that the first block $\bPhi_{\bq_{\rm dep}}$ is nonsingular. Here $\Permutation \in \mathds{R}^{n \times n}$ is a permutation matrix, obtained by permuting rows of identity matrix; the multiplication $\bPhi_\bq \cdot \Permutation$ performs a permutation of the columns of $\bPhi_\bq$.

By the implicit function theorem one can partition locally the position state variables into independent coordinates $\bq_{\rm dof}  \in \mathds{R}^{f}$ (the local `degrees of freedom' of the system) and dependent coordinates $\bq_{\rm dep} \in \mathds{R}^{m}$, and solve for the dependent ones in terms of the degrees of freedom:
\[
\bPhi(t,\bq) = \bzero \quad \textnormal{and} \quad \bPhi_{\bq_{\rm dep}}(t,\bq)~ \textnormal{nonsingular}
\quad \Rightarrow \quad
\bq_{\rm dep} = \zeta\bigl( t,\bq_{\rm dof}  \bigr).
\]
This induces a  corresponding local partitioning of the state variables into independent components $\bq_{\rm dof},\bv_{\rm dof}  \in \mathds{R}^{f}$ and dependent components $\bq_{\rm dep},\bv_{\rm dep} \in \mathds{R}^{m}$:
\begin{equation}
\label{eq:coordinate-partitioning}
\Permutation \cdot \bq =  \begin{bmatrix} \Permutation_{\rm dep} \\ \Permutation_{\rm dof}\end{bmatrix}  \cdot \bq = \begin{bmatrix} \bq_{\rm dep} \\ \bq_{\rm dof} \end{bmatrix}, \qquad
\Permutation \cdot \bv =  \begin{bmatrix} \Permutation_{\rm dep}  \\ \Permutation_{\rm dof}  \end{bmatrix} \cdot \bv = \begin{bmatrix} \bv_{\rm dep} \\ \bv_{\rm dof} \end{bmatrix},
\end{equation}
where $\Permutation_{\rm dep} \in \mathds{R}^{m \times n}$ and $\Permutation_{\rm dof} \in \mathds{R}^{f \times n}$ consist the first $m$ and the last $f$ rows of $\Permutation$, respectively.
Let:
\begin{equation}
\label{eq:R-matrix}
\Rez :=  -\bPhi_{\bq_{\rm dep}}^{-1}\,\bPhi_{\bq_{\rm dof}}  \in \mathds{R}^{m \times f}.
\end{equation}
The velocity constraint equation \eqref{eq:ConstraintsEq-velocity} becomes:
\begin{equation}
\label{eq:ConstraintsEq-velocity-dof}
 \bPhi_{\bq_{\rm dep}}\, \bv_{\rm dep} + \bPhi_{\bq_{\rm dof}}\, \bv_{\rm dof} = -\bPhi_t
\quad \Rightarrow \quad
 \bv_{\rm dep} = -\bPhi_{\bq_{\rm dep}}^{-1}\, \bigl( \bPhi_{\bq_{\rm dof}}\, \bv_{\rm dof} +  \bPhi_t \bigr)
 = \Rez\,\bv_{\rm dof}  - \bPhi_{\bq_{\rm dep}}^{-1}\,\bPhi_t. 
\end{equation}
Similarly, the acceleration constraint equation \eqref{eq:ConstraintsEq-acceleration} becomes:
\begin{equation}
\label{eq:ConstraintsEq-acceleration-dof}
\bPhi_{\bq_{\rm dep}}\, \dbv_{\rm dep} + \bPhi_{\bq_{\rm dof}}\, \dbv_{\rm dof} = \Faccel
\quad \Rightarrow \quad
\dbv_{\rm dep} = -\bPhi_{\bq_{\rm dep}}^{-1}\, \bigl( \bPhi_{\bq_{\rm dof}}\, \dbv_{\rm dof} - \Faccel \bigr)
= \Rez\,\dbv_{\rm dof} + \bPhi_{\bq_{\rm dep}}^{-1}\,\Faccel.
\end{equation}
From \eqref{eqn:split-constraint-jacobian}, \eqref{eq:coordinate-partitioning}, and \eqref{eq:R-matrix} we have that:
\[
\bPhi_{\bq_{\rm dof}}  = \bPhi_\bq \cdot \Permutation_{\rm dof}^{\rm T} \in \mathds{R}^{m \times f}, \quad
\bPhi_{\bq_{\rm dep}}  = \bPhi_\bq \cdot \Permutation_{\rm dep}^{\rm T} \in \mathds{R}^{m \times m}, \quad
\Rez = -\bigl(  \bPhi_\bq \cdot \Permutation_{\rm dep}^{\rm T}  \bigr)^{-1}\cdot \bPhi_\bq \cdot \Permutation_{\rm dof}^{\rm T}.
\]

%
\subsection{Representation of constrained hybrid  multibody systems}

The hybrid dynamics of a constrained mechanical system  refers to the smooth system defined in Section ~\ref{sec:multibody-smooth} subjected to a finite number of events, as discussed in Definition \ref{def:time-of-invent}. Each event \eqref{eq:event_function} happening at the `time of event' $\timp$ introduces a kink in the trajectory of the mechanical system. At each event the velocity state vector of an {\it unconstrained} system undergoes a jump \eqref{eq:jump-in-velocity} that can be arbitrary, i.e., can be described by any smooth function $h(\cdot)$. In case of a {\it constrained} system we need a more comprehensive understanding of the event.

\begin{definition}[Characterization of an event for constrained multibody systems]
\label{def:characterize-impulsive jump event} 
During an event at time $\timp$ a constrained mechanical system undergoes a sudden change in state characterized as follows:
\begin{itemize}
\item The constraints may change at the time of event (e.g., when a walking humanoid robot changes its supporting foot at each step). Consequently, the position constraint function \eqref{eq:ConstraintsEq-position}  changes from $\bPhi^- : \mathds{R}^{1+n+p} \to \mathds{R}^{m-}$ before invent to $\bPhi^+ : \mathds{R}^{1+n+p} \to \mathds{R}^{m+}$ after invent:
\begin{equation*}
\bPhi^-(t,\bq,\brho) \stackrel{\rm event}{\longrightarrow} \bPhi^+(t,\bq,\brho).
\end{equation*}
The two constraint functions are different, and in particular the number of constraints can differ, $m^+ \neq m^-$. 
\item The Jacobians of the position constraints  before and after the invent have full row ranks at the invent configuration $\qtimp$:
\[
\operatorname{rank}\big(\bPhi_\bq^-(t,\bq,\brho)\big)=m^{-}, \quad 
\operatorname{rank}\big(\bPhi_\bq^+(t,\bq,\brho)=m^{+}.
\]
\item Since the constraints can be different after and before the event, the partitions of variables into independent and dependent can also differ. We denote by $\square_\textnormal{dof-},\square_\textnormal{dep-}$ the independent and dependent components before the event, and by $\square_\textnormal{dof+},\square_\textnormal{dep+}$ the independent and dependent components after the event:
\begin{equation}
\label{eq:velocity-partitioning}
\Permutation^{-} \cdot \bv = \begin{bmatrix} \bv_\textnormal{dep-} \\ \bv_\textnormal{dof-} \end{bmatrix} \in \Re^n, \quad \bv_\textnormal{dof-} \in \Re^{f-}, 
\qquad
\Permutation^{+} \cdot \bv =   \begin{bmatrix} \bv_\textnormal{dep+} \\ \bv_\textnormal{dof+} \end{bmatrix} \in \mathds{R}^n, 
\quad\bv_\textnormal{dof+}\in \mathds{R}^{f+}.
\end{equation}
Here $\Permutation^{-}$ and $\Permutation^{+}$ are the permutation matrices that select the dependent and independent coordinates before and after the event, respectively. The dimensions of the velocity degrees of freedom vectors are  $f^{-} = n-m^{-}$ and $f^{+} = n-m^{+}$ before and after the event, respectively.

\item The generalized position state variables remain the same \eqref{eq:jump-in-position}, i.e., $\bq\afterimp = \bq\beforeimp = \qtimp$. Consequently, the state at the time of event $\qtimp$ need to satisfy both constraint functions:
\begin{equation}
\bPhi^-\beforeimp:=\bPhi^-\left(\timp,\qtimp, \brho \right) = 0,
\qquad
\bPhi^+\timpplus:=\bPhi^+\left(\timp,\qtimp, \brho \right) = 0.
\end{equation}
At the invent the system moves from one constraint manifold to another, and $\qtimp$ is on the intersection of the two manifolds.

\item The jump in velocity from right before the invent to right after the invent is defined in terms of the independent components, i.e., in terms of the velocity degrees of freedom:
\begin{equation}
\label{eq:jump-in-velocity-constrained-dof}
\bv_\textnormal{dof+}\afterimp =h \Big({\timp} , \qtimp, \bv_\textnormal{dof-}\beforeimp, {\brho} \Big), \qquad
h : \mathds{R}^{1+n+f^-+p} \to  \mathds{R}^{f^+}.
\end{equation}
The jump function \eqref{eq:jump-in-velocity-constrained-dof} is assumed to be smooth. Note that its formulation is not unique, since it depends on the selections of the degrees of freedom that are not unique. 

\item The velocity state vectors satisfy the velocity kinematic constraints \eqref{eq:ConstraintsEq-velocity-dof}. Consequently, the jumps in velocity \eqref{eq:jump-in-velocity} cannot be arbitrary for the dependent components. The dependent components of velocity are obtained from solving the velocity constraints \eqref{eq:ConstraintsEq-velocity-dof}:
\begin{equation}
\label{eq:jump-in-velocity-constrained-dep}
\begin{split}
 \bv_{\rm dep+}\afterimp &= -\left(\bPhi^+_{\bq_{\rm dep+}}\afterimp\right)^{-1}\cdot \left( \bPhi^+_{\bq_{\rm dof+}}\afterimp\, \bv_{\rm dof+}\afterimp  +  \bPhi^+_t\afterimp \right) \\
 &= \Rez^+\afterimp\,  \bv_{\rm dof+}\afterimp -\left(\bPhi^+_{\bq_{\rm dep+}}\afterimp\right)^{-1}\cdot   \bPhi^+_t\afterimp.
 \end{split}
\end{equation}
Here $\Rez^\pm$ are the matrices corresponding to the constraints $\bPhi^\pm$.
\end{itemize}
\end{definition} 

\begin{remark}[Collision events]
\label{rem:jump-in-collision}
The proposed formalism \eqref{eq:jump-in-velocity-constrained-dof}--\eqref{eq:jump-in-velocity-constrained-dep} covers the case of elastic contact/collision/impact without change in the set of constraint equations,  $\bPhi^+\equiv\bPhi^-$. The impulsive (external) contact forces act to change the independent components of the velocity state \eqref{eq:jump-in-velocity-constrained-dof}.
\end{remark}	

%

\begin{remark}[Hybrid DAE jump formulation]
\label{rem:jump-in-robotics}
The proposed formalism \eqref{eq:jump-in-velocity-constrained-dof}--\eqref{eq:jump-in-velocity-constrained-dep} also covers the case where the event consists solely of a change of constraints $\bPhi^+\ne\bPhi^-$, without any external force to modify the independent velocities. This type of event appears mainly in the humanoid robotics field where general and relative coordinates are used and inelastic collisions are considered. A popular approach in robotics  is to solve for the DAE involving impulsive forces in the constraints at the time of event  \cite{Kolathaya2016, Posa2013}:
\begin{subequations}
\label{eq:DAE-impulse}
\begin{equation}
\label{eq:DAE-impulse-1}
\begin{bmatrix}
{\Mass\atimp} & (\dPhidq^+)^{\rm T}\atimp \\
\dPhidq^+\atimp & \bzero
\end{bmatrix}
\cdot
\begin{bmatrix}
\vplus \\ \delta\lambda
\end{bmatrix}
=
\begin{bmatrix}
 \Mass\atimp \cdot \vminus \\
-\bPhi_t^+\atimp 
\end{bmatrix},
\end{equation}
or equivalently,
\begin{equation}
\label{eq:DAE-impulse-2}
\begin{bmatrix}
\vplus \\ \delta\lambda
\end{bmatrix}
=
\begin{bmatrix}
{\Mass\atimp} & (\dPhidq^+)^{\rm T}\atimp \\
\dPhidq^+\atimp & \bzero
\end{bmatrix}^{-1}
\cdot
\begin{bmatrix}
 \Mass\atimp \cdot \vminus \\
-\bPhi_t^+\atimp 
\end{bmatrix}=
\begin{bmatrix}
\fdaeimpv {\left(\,\timp,\,\qtimp ,\,\vminus,\,\brho \, \right)} \\
\fdaeimpl {\left(\,\timp,\,\qtimp ,\,\vminus,\,\brho \, \right)}
\end{bmatrix}.
\end{equation}
\end{subequations}
Here $\vplus$ contains both the independent and dependent coordinates. We see that the second equation in \eqref{eq:DAE-impulse-1} automatically imposes the velocity constraint \eqref{eq:ConstraintsEq-velocity}. 

Our formalism covers this approach by defining the jump function given by Eq. \eqref{eq:jump-in-velocity-constrained-dof} as:
\[
\bv_{\rm dof+}\afterimp = \Permutation_{\rm dof+}\,\fdaeimpv {\left(\,\timp,\,\qtimp ,\,\vminus,\,\brho \, \right)}
=: h\left(\,\timp,\,\qtimp ,\,\bv_{\rm dof-}\beforeimp,\,\brho \, \right).
\]
%
\end{remark} 

\subsection{The jump in the sensitivity of the position state vector}
The jump conditions at the time of event in the sensitivity state vector for a constrained rigid multibody involve finding the sudden change in values of the sensitivity with respect to the system parameters $\brho$ of the position and the dependent and independent velocity state variables due the impulsive jump of the independent velocity state variables. 

\begin{remark}[Partitioning of sensitivity matrices]
The partitioning of state variables into dependent and independent \eqref{eq:velocity-partitioning} induces a similar partitioning of the sensitivity matrices \eqref{eq:sensitivity-of-solutions}:
\begin{equation}
\label{eq:sensitivity-partitioning}
\Permutation \cdot \bQ = \begin{bmatrix} \bQ_\textnormal{dep} \\ \bQ_\textnormal{dof} \end{bmatrix} \in \mathds{R}^{n \times p}, \quad \bQ_\textnormal{dof} \in \mathds{R}^{f \times p}, \qquad
\Permutation \cdot \bV = \begin{bmatrix} \bV_\textnormal{dep} \\ \bV_\textnormal{dof} \end{bmatrix} \in \mathds{R}^{n \times p}, 
\quad\bV_\textnormal{dof} \in \mathds{R}^{f \times p}.
\end{equation}
Differentiation of the position constraint equation \eqref{eq:ConstraintsEq-position} with respect to the system parameters $\brho$ gives:
\[
\bzero  = \frac{d \bPhi(t,\bq(t,\brho),\brho)}{d \brho} = \bPhi_\bq \cdot \bQ +  \bPhi_\brho = 
\bPhi_{\bq_\textnormal{dep}} \cdot \bQ_\textnormal{dep} + \bPhi_{\bq_\textnormal{dof}} \cdot \bQ_\textnormal{dof} +  \bPhi_\brho,
\]
and therefore:
\begin{equation}
\label{eq:dependent-Q}
\bQ_\textnormal{dep}= -\bPhi_{\bq_\textnormal{dep}}^{-1}\, \left( \bPhi_{\bq_\textnormal{dof}} \cdot \bQ_\textnormal{dof} +\bPhi_\brho \right) = \Rez \cdot \bQ_\textnormal{dof} - 
 \bPhi_{\bq_\textnormal{dep}}^{-1}\, \bPhi_\brho.
\end{equation}
Similarly, differentiation of the velocity constraint equation \eqref{eq:ConstraintsEq-velocity} with respect to the system parameters gives:
\begin{eqnarray*}
\bzero &=& \frac{d}{d \brho} \left( \dPhidq(t,\bq(t,\brho),\brho) \, \bv(t,\brho) + \bPhi_t(t,\bq(t,\brho),\brho) \right) \\
&=& \bPhi_{\bq} \cdot \bV + (\bPhi_{\bq,\bq}\,\bv + \bPhi_{\bq,t}) \cdot \bQ + \bPhi_{\brho,\bq}\,\bv + \bPhi_{\brho,t} \\
&=& \bPhi_{\bq_\textnormal{dep}} \cdot \bV_\textnormal{dep} + \bPhi_{\bq_\textnormal{dof}} \cdot \bV_\textnormal{dof}
 + \big(\bPhi_{\bq,\bq}\,\bv +  \bPhi_{\bq,t} \big) \cdot \bQ 
+  \bPhi_{\brho,\bq} \, \bv  + \bPhi_{\brho,t},
 \end{eqnarray*}
and therefore:
\begin{equation}
\label{eq:dependent-V}
\bV_\textnormal{dep} = \Rez \cdot \bV_\textnormal{dof} 
- \bPhi_{\bq_\textnormal{dep}}^{-1}
 \left(\big(\bPhi_{\bq,\bq}\,\bv +  \bPhi_{\bq,t} \big) \cdot \bQ 
 +  \bPhi_{\brho,\bq} \, \bv  + \bPhi_{\brho,t}
 \right).
\end{equation}

\end{remark}

\begin{remark}[Sensitivity of the time of event for constrained systems]
The time of event depends only on the position state and on the event function \eqref{eq:event_function}. Consequently, the sensitivity of the time of event for constrained systems is the same as for unconstrained systems, and is given by Eqn.  \eqref{eq:sensitivity-time-of-invent} in Theorem \ref{thm:sensitivity-time-of-invent}.
\end{remark}

\begin{theorem}[Jump in position sensitivity for constrained system] 
\label{thm:jump-in-position-sensitivity-constrained}
Let $\bQ\afterimp$ and $\bQ\beforeimp \in \mathds{R}^{n{\times}p}$  be the sensitivities of the generalized position state vectors right after and right before the event, respectively.  The independent components of the sensitivity of the generalized positions right after the event are:
\begin{subequations}
\label{eq:jump-position-sensitivity-constrained}
\begin{equation}
\label{eq:jump-position-sensitivity-constrained-dof}
\bQ_\textnormal{dof+}\afterimp = \bQ_\textnormal{dof+}\beforeimp  - \bigg( \bv_\textnormal{dof+}\afterimp - \bv_\textnormal{dof+}\beforeimp \bigg) \cdot \frdtdrho.
\end{equation}
The dependent components of the sensitivity of the generalized positions right after the event are given by equation \eqref{eq:dependent-Q}, using the after-event constraints:
\begin{equation}
\label{eq:jump-position-sensitivity-constrained-dep}
\bQ_\textnormal{dep+}\afterimp =  \Rez^+\afterimp \cdot \bQ_\textnormal{dof+}\afterimp - \left. \left(\bPhi^+_{\bq_\textnormal{dep+}}\afterimp\right)^{-1}\, \bPhi^+_\brho \right\afterimp.
\end{equation}
\end{subequations}
\end{theorem}  
\begin{pf} 
The proof of the jump in the independent coordinates \eqref{eq:jump-position-sensitivity-constrained-dof} follows closely the proof of Theorem \ref{thm:jump-in-position-sensitivity}. The equation for dependent coordinates \eqref{eq:jump-position-sensitivity-constrained-dep} follows from the linearized position constraint equation \eqref{eq:dependent-Q}.
\qed
\end{pf}

\begin{remark}
From \eqref{eqn:split-constraint-jacobian} and \eqref{eq:coordinate-partitioning} we can rewrite \eqref{eq:jump-position-sensitivity-constrained-dof} as:
\begin{equation}
\Permutation^+_\textnormal{dof+} \cdot \bigg(\bQ\afterimp - \bQ\beforeimp  \bigg)
= - \Permutation^+_\textnormal{dof+} \cdot \bigg( \bv\afterimp - \vminus\bigg) \cdot \frdtdrho.
\end{equation}
\end{remark}

\subsection{The jump in the sensitivity of the velocity state vector}
%
\begin{theorem}[Jump in velocity sensitivity for a constrained system]  
\label{thm:jump-in-velocity-sensitivity-constrained}
Let $\bV\afterimp, \bV\beforeimp \in \mathds{R}^{n{\times}p}$ be the sensitivity matrices of the generalized velocity state vectors after and before the invent, respectively. The independent coordinates of the velocity sensitivities right after the event are given by:
\begin{subequations}
\label{eq:jump-in-velocity-sensitivity-constrained}
%
\begin{equation}
\label{eq:jump-in-velocity-sensitivity-constrained-dof}
\begin{split}
\bV_\textnormal{dof+}\afterimp&= h_{\bq}\beforeimp \cdot \qrhominus 
+h_{\bv_\textnormal{dof-}}\beforeimp \cdot  \bV_\textnormal{\rm dof-}\beforeimp 
+\Bigl( h_{\bq}\beforeimp\cdot \vminus-\ddbq_\textnormal{dof+}\afterimp +h_{\bv_\textnormal{dof-}}\beforeimp\cdot \ddbq_\textnormal{dof-}\beforeimp {+ h_t\beforeimp} \Bigr) \cdot \frdtdrho
{+ h_\brho\beforeimp},
\end{split}
\end{equation}
where the Jacobians of the jump function are:
\[
\begin{split}
h_{\bq}\beforeimp &:= \frac{\partial h}{\partial\bq}\big( \bq\atimp,\bv_\textnormal{dof-}\beforeimp, \brho \big) \in \Re^{f^+ \times n},
\qquad
h_{\bv_\textnormal{dof-}}\beforeimp := \frac{\partial h}{\partial\bv_\textnormal{dof-}}\big( \bq\atimp,\bv_\textnormal{dof-}\beforeimp, \brho \big) \in \Re^{f^+ \times f^-}.
\\
h_t\beforeimp &:= \frac{\partial h}{\partial t}\big( \bq\atimp,\bv_\textnormal{dof-}\beforeimp, \brho \big) \in \Re^{f}, \qquad \;\;\;\;
h_\brho\beforeimp := \frac{\partial h}{\partial\brho}\big( \bq\atimp,\bv_\textnormal{dof-}\beforeimp, \brho \big)\in \Re^{f \times p}.
\end{split}
\]
The dependent components  of the velocity sensitivities right after the event are calculated via \eqref{eq:dependent-V}, using the after-event constraints:
\begin{equation}
\label{eq:jump-in-velocity-sensitivity-constrained-dep}
\bV_\textnormal{dep+}\afterimp =
\left. - \big( \bPhi^+_{\bq_\textnormal{dep+}} \afterimp \big)^{-1}
 \left(
\bPhi^+_{\bq_\textnormal{dof+}} \cdot \bV_\textnormal{dof+} 
 + \big(\dPhidqqplus \,\bv +  \dPhidtdqplus \big) \cdot \bQ 
 +  \dPhidqdrhoplus \, \bv  + \dPhidtdrhoplus
 \right)\right\afterimp.
\end{equation}
\end{subequations}
\end{theorem}
\begin{pf}
The proof of the jump in the independent coordinates \eqref{eq:jump-in-velocity-sensitivity-constrained-dof} follows closely the proof of Theorem \ref{thm:jump-in-velocity-sensitivity}. The equation for dependent coordinates \eqref{eq:jump-in-velocity-sensitivity-constrained-dep} follows from the linearized velocity constraint equation \eqref{eq:dependent-V}.
\qed
\end{pf} 

\subsection{The jump in the sensitivity of the velocity state vector using the hybrid DAE jump formulation}

Consider the case of a sudden change in constraints discussed in Remark \ref{rem:jump-in-robotics}. 
The jump in the velocity sensitivity for constrained system due to impulsive forces is determined as follows:
\[
\begin{split}
\begin{bmatrix}
{\Mass\afterimp} & \dPhidqplus^{\rm T}\afterimp \\
\dPhidqplus\afterimp & \bzero
\end{bmatrix}
\cdot
\begin{bmatrix}
\dqrhoplus \\ \delta\Lambda
\end{bmatrix}
=&
-\begin{bmatrix}
 {\Mass\afterimp}_{\bq}\afterimp \cdot (\vplus - \vminus) + \dPhidqqplus^{\rm T}\afterimp \cdot \delta\lambda \\
 \dPhidqqplus\afterimp \cdot \vplus
\end{bmatrix} \cdot \qrhoplus +
\begin{bmatrix}
\Mass\afterimp \\ 
\bzero
\end{bmatrix} \cdot \dqrhominus \\&
-\begin{bmatrix}
 {\Mass}_{\brho}\afterimp \cdot \vplus +  \dPhidqdrhoplus ^{\rm T}\afterimp \cdot \delta\lambda \\
 \dPhidqdrhoplus\afterimp \cdot \vplus + \dPhidtdqplus\afterimp \cdot \vminus
 + \dPhidtdvplus\afterimp \cdot \vminus  + \dPhidtdrhoplus\afterimp \cdot \vminus
\end{bmatrix},
\end{split}
\]
or equivalently:
\[
\begin{split}
\begin{bmatrix}
\dqrhoplus \\ \delta\Lambda
\end{bmatrix}
=&
-
\begin{bmatrix}
{\Mass} & \dPhidqplus^{\rm T} \\
\dPhidqplus & \bzero
\end{bmatrix}^{-1}
\begin{bmatrix}
 {\Mass}_{\bq} \cdot (\vplus - \vminus) + \dPhidqqplus^{\rm T} \cdot \delta\lambda \\
 \dPhidqqplus \cdot \vplus
\end{bmatrix}\cdot \qrhoplus +
\begin{bmatrix}
{\Mass} & \dPhidqplus^{\rm T} \\
\dPhidqplus & \bzero
\end{bmatrix}^{-1}
\begin{bmatrix}
\Mass \\ 
\bzero
\end{bmatrix} \cdot \dqrhominus
\\& -
\begin{bmatrix}
{\Mass} & \dPhidqplus^{\rm T} \\
\dPhidqplus & \bzero
\end{bmatrix}^{-1}
\begin{bmatrix}
 {\Mass}_{\brho} \cdot \vplus +  \dPhidqdrhoplus^{\rm T} \cdot \delta\lambda \\
 \dPhidqdrhoplus \cdot \vplus  + \dPhidtdrhoplus \cdot \vminus
\end{bmatrix}
-
\begin{bmatrix}
{\Mass} & \dPhidqplus^{\rm T} \\
\dPhidqplus & \bzero
\end{bmatrix}^{-1}
\begin{bmatrix}
 \bzero \\
  \dPhidtdqplus \cdot \vminus
 + \dPhidtdvplus \cdot \vminus  
\end{bmatrix},
\end{split}
\]
which simplifies to:
\[
\begin{bmatrix}
\dqrhoplus \\ \delta\Lambda
\end{bmatrix}
=
\fun_\bq^{\scalebox{0.6}{\rm DAE-imp}} \cdot \qrhoplus 
+\fun_{\vminus}^{\scalebox{0.6}{\rm DAE-imp}} \cdot \dqrhominus
+ \fun_{\brho}^{\scalebox{0.6}{\rm DAE-imp}} 
+ \fun_{t}^{\scalebox{0.6}{\rm DAE-imp}} 
\]

\subsection{The jump in the sensitivity of the Lagrange multipliers}

\begin{remark}
\label{rem:jump in DAE Lagrange multipliers}
When the DAE formalism is selected to model the smooth dynamics of a constrained mechanical system, the jump in the sensitivity of the Lagrange multipliers \eqref{eq:sensitivity-of-multipliers} from $\Lambda\beforeimp \to \Lambda\afterimp$  at the time of event is:
\begin{equation}
\Lambda_i\afterimp  =\Lambda_i\beforeimp  +\fdaelbq \afterimp \; \bQ_i\afterimp +\fdaelbv\afterimp \; \bV_i\afterimp +\fdaelbrhoi\afterimp, \quad i=1,\dots,p
\end{equation}
 \end{remark}

\begin{remark} 
\label{rem:jump in estimated Lagrange multipliers}
When the ODE penalty formalism is selected to model the smooth dynamics of a constrained mechanical system, the jump in the sensitivity of the estimated Lagrange multipliers \eqref{eq:sensitivity-of-estimated-multipliers} from $\Lambda^{*}\beforeimp \to \Lambda^{*}\afterimp$  at the time of event is:
\begin{equation}
\Lambda_i^{*}\afterimp  =\Lambda_i^{*}\beforeimp  +\blambda_{\bq}^{*} \afterimp \; \bQ_i\afterimp +\blambda_{\bv}^{*}\afterimp \; \bV_i\afterimp +\blambda_{\brho_i}^{*}\afterimp, \quad i=1,\dots,p
\end{equation}
\end{remark}

\subsection{The sensitivity of the cost function for hybrid systems}

\begin{remark}
\label{rem:jump-in-the-cost-function}	
The formalism that computes the sensitivities of the cost function with respect to parameters for hybrid systems does not change from the formalism presented for smooth systems illustrated in Remark \ref{rem:sensitivities-of-the-cost-function}. Indeed, the sensitivities of the cost function sum all the sensitivities of the trajectories and the quadrature variables evaluated at the final time. Any jump in the sensitivities of the trajectories and  quadrature variables were anteriorly computed. The jump in the sensitivities of the quadrature variables are given by \eqref{eq:jump-in-quadrature-sensitivity}.
\end{remark}

\section{Direct sensitivity analysis for constrained mechanical systems with transition functions}
\label{sec:Direct sensitivity analysis with transition functions}

The transition function refers to a sudden change of the governing function or vector field. In this section we discuss about direct sensitivity analysis for constrained mechanical with jump-discontinuity in the acceleration caused by a sudden change of the equation of motions at the time of event. 
\begin{definition}[Change of EOM at the time of event]  
Unlike previous methodology, the ODE penalty formulation Eq.~\eqref{eq:EOM-ODE-penalty} incorporates the kinematic constraints (position, velocity, and acceleration constraint equations) into the equation of motions and stabilize them over time. Therefore, any change in the set of kinematic constraints involves a change in the equation of motions, thus, a change in the acceleration vector {blue}{(or right-hand side function)}. Because  the ODE penalty formulation is a control based constraint stabilization method, the position constraint is not satisfied exactly right after the sudden change in the set of kinematic constraints:
\begin{equation}
\bPhi^-\beforeimp = 0 , \quad \bPhi^+\timpplus \neq 0.
\end{equation}
This differs from  \eqref{eq:DAE-impulse} as there are no instantaneous kinematic jump in the velocity state variable. 
\end{definition}

\begin{theorem}[Jump in the velocity sensitivity for constrained system due to the change of equation of motions]  
\label{thm:jump-in-velocity-sensitivity-constrained-EOMchanged}
Let $\bV\afterimp, \bV\beforeimp \in \mathds{R}^{n{\times}p}$ be the sensitivity matrices of the generalized velocity state vectors after and before the invent, respectively. Let the event characterized as a change of equation of motions due to the change of constraints including in the equation of motions. The sensitivities of the independent  velocities right after the event are given by
\begin{equation}
\bV_\textnormal{dof-}\afterimp =\bV_\textnormal{dof-}\beforeimp + \left(  \ddbq_\textnormal{dof-}\beforeimp - \ddbq_\textnormal{dof-}\afterimp \right) \cdot \frdtdrho.
\end{equation}
\end{theorem}
\begin{proof}
The proof of the jump in the independent coordinates (58a) follows closely the proof of Theorem \ref{thm:ODE-change-in-acceleration}.
\end{proof}
\begin{remark}
The sensitivities of the dependent  velocities right after the event are given by \eqref{eq:jump-in-velocity-sensitivity-constrained-dep}. As well, the sensitivities of the position right after the event for the independent and dependent variables are given by \eqref{eq:jump-position-sensitivity-constrained-dof} and  \eqref{eq:jump-position-sensitivity-constrained-dep}, respectively.
\end{remark}

\begin{remark}
The sudden change in the equation of motions at the event time is caused by a sudden change of forces acting on the system, such as constraint forces, friction forces, or a change of masses. The proposed formalism to calculate the jump conditions for systems{blue}{ with discontinuous right-hand sides} remains valid for any type of change of the equation of motions. 
\end{remark}

\begin{remark} The proposed formalism in calculating the jump conditions for systems with jerk discontinuity incorporates  Remarks  \ref{rem:jump in estimated Lagrange multipliers} and \ref{rem:jump-in-the-cost-function}. 
\end{remark}

\section{Case study: sensitivity analysis of a five-bar mechanism}
\label{sec:numerics}
The five-bar mechanism with two degrees of freedom, shown in Fig.~\ref{Fig_fivebar}, is used as a case study to illustrate the sensitivity analysis approach for hybrid constrained multibody systems developed herein.

The mechanism has five revolute joints located at points A, 1, 2, 3, and B; the masses of the bars are $m_1=1 \; kg$, $m_2=1.5 \; kg$, $m_3=1.5 \; kg$, $m_4=1 \; kg$; the polar moments of inertia are assumed to be ideal, with uniform distribution of mass; the two springs have stiffness coefficients of $k_1=k_2=100 \; N/m$ and natural lengths of $L_{01}=2.2360 \; m$ and $L_{02}=2.0615 \; m$.

The state vector includes the natural coordinates of the point 1, 2, and 3 of the mechanism with $\bq= \big[  \bq^{\rm T}_{1} ~ \bq^{\rm T}_{2} ~ \bq^{\rm T}_{3} \big]^{\rm T}$ where  $\bq_{2}= \big[ x_{2} ~ y_{2}  \big]^{\rm T}$ represents the independent coordinates, and $\bq_{1}= \big[  x_{1} ~ y_{1}  \big]^{\rm T}$ with  $\bq_{3}= \big[ x_{3} ~ y_{3} \big]^{\rm T}$ represent the dependent coordinates. This implies that the coordinates of point 2 are the DOF of the system on which the  coordinates of points 1 and 3 are dependent. The dependent coordinates are solved using the constraints of the system. The constraints are defined according to the fixed lengths between each set of points, as follows:
\begin{align}
	\label{eq:Constraints 5 bar}
    {\bPhi} &= \begin{bmatrix}
           \|{q_A-q_1}\|^2  -L_{A1}^2\\
			\|{q_2-q_1}\|^2  -L_{21}^2\\
           \|{q_3-q_2}\|^2  -L_{32}^2\\
			\|{q_B-q_3}\|^2  -L_{B3}^2\\
         \end{bmatrix}=0,
\end{align}
with the lengths $L_{A1}=L_{B3}=1.4142 \; m$ ; $L_{A1}=L_{B3}=1.8027 \; m $ and the ground points $\bq_{A}= \big[ \begin{array}{c c} -0.5 & 0 \end{array} \big]^{\rm T}$; $\bq_{B}= \big[ \begin{array}{c c} 0.5 & 0 \end{array} \big]^{\rm T}$. Finally, the point 2 of the five-bar mechanism hits the ground at -2.35 m along the vertical y axis. The event (touching the ground) is simulated through the event function $r(\cdot)$ described in Eq.~\eqref{eq:event_function}. Once the event is detected, the vertical velocity of point 2 jumps to its opposite value, while its horizontal velocity remains the same. 

The system of equation \eqref{eq:MBD canonical ODE system} is simulated with a time span of five seconds. Fig.~\ref{fig2} shows the residuals of the constraint equations. The position constraints are satisfied within an error of $10^{-6}$, while the velocity constraints within and error of $10^{-5}$, which is satisfactory.
\begin{figure} [H] 
\centering
	\begin{subfigure}[t]{.49\textwidth}
	\centering
	\raisebox{-4.3mm}{\includegraphics[width=0.9\textwidth]{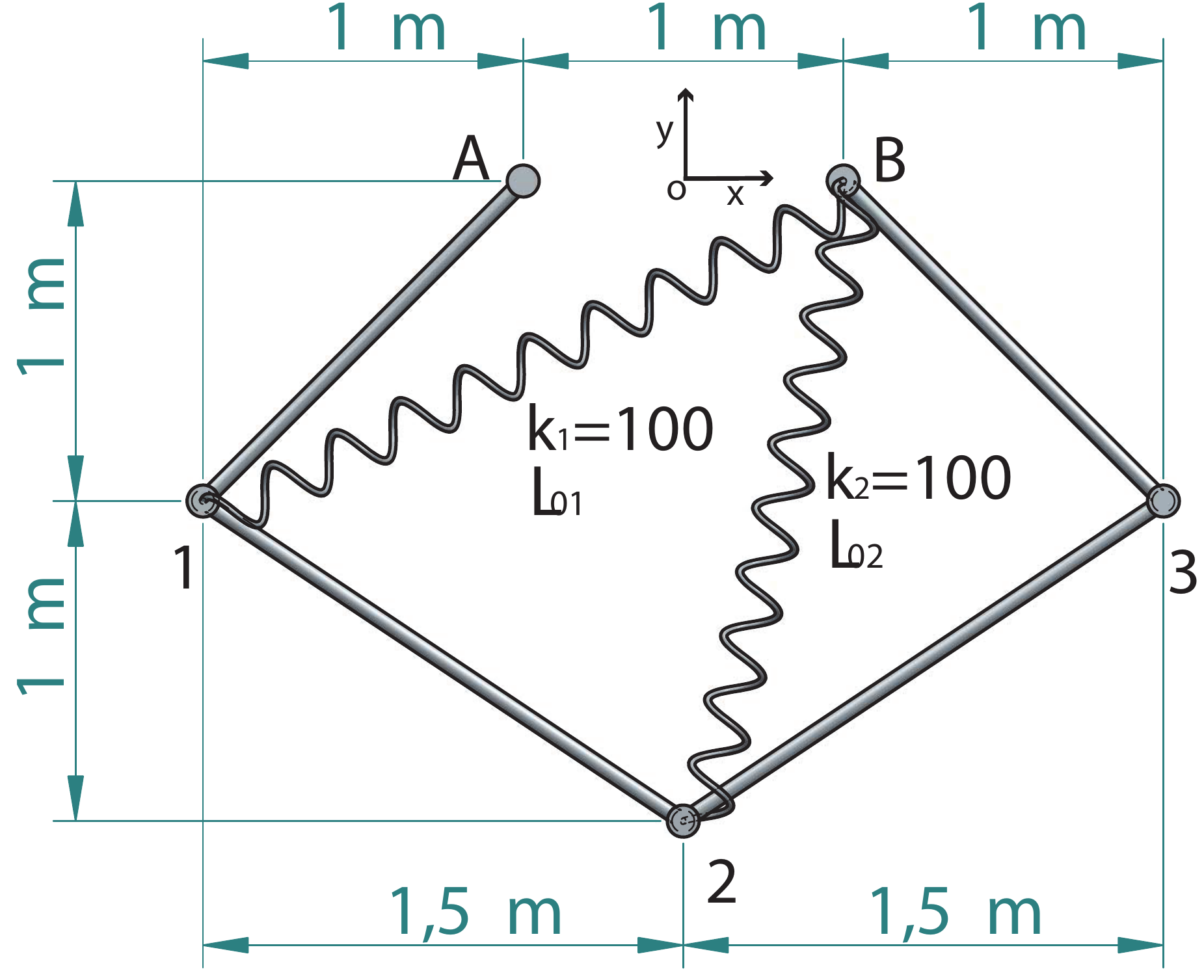}}
	\caption{Diagram of the five-bar mechanism}
	\label{Fig_fivebar}
	\end{subfigure}
	\begin{subfigure}[t]{.49\textwidth}
	\centering
	\includegraphics[width=\textwidth]{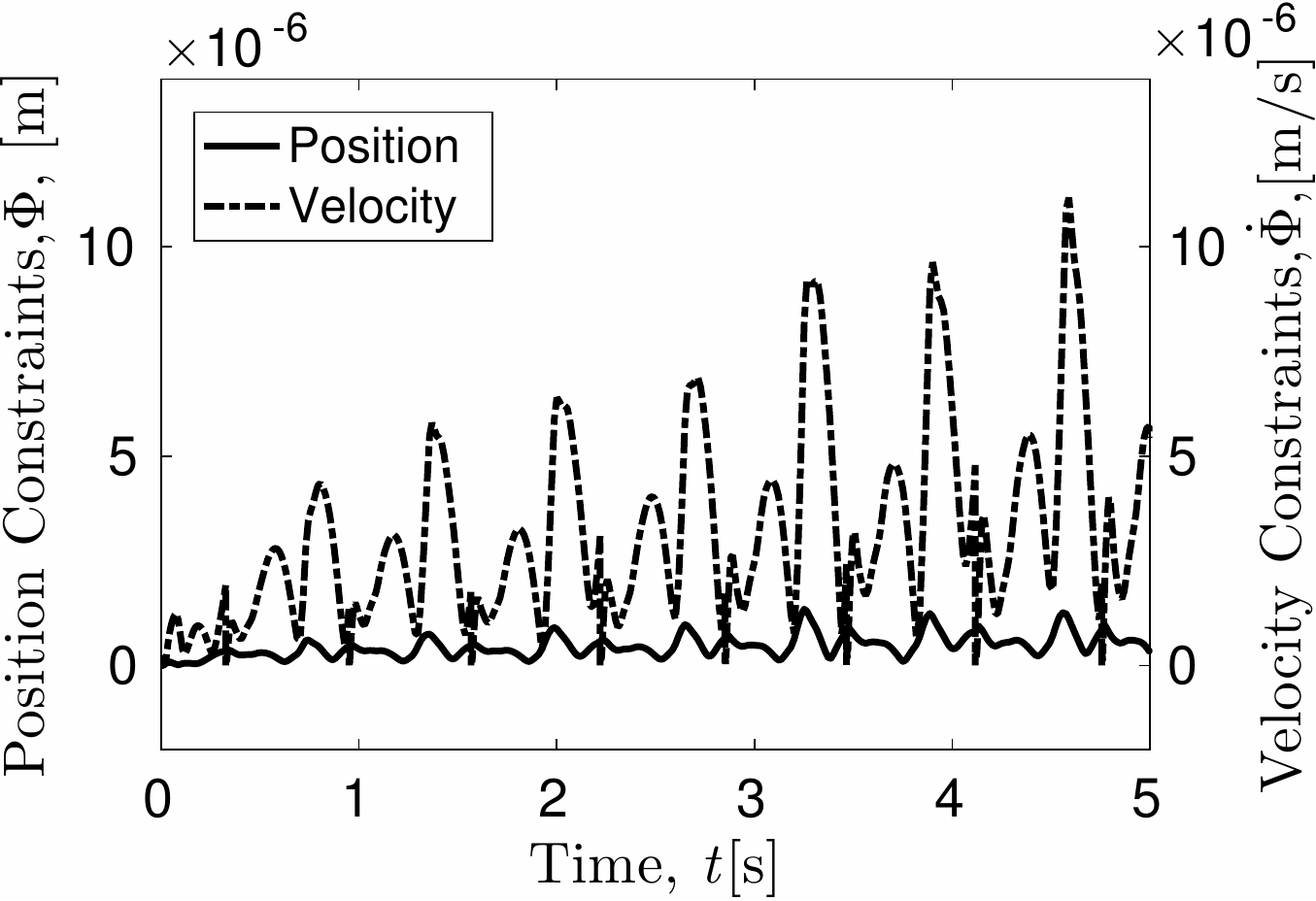}
	\captionsetup{margin=1cm}
	\caption{The position and the velocity constraint residuals for the five-bar mechanism}
	\label{fig2}
	\end{subfigure}
\caption{Structure of the five-bar mechanism}
\label{sfig1}
\end{figure}
The trajectories of the position and velocity of point 2 of the five-bar mechanism along the vertical y axis are shown in Fig.~\ref{fig3} and Fig.~\ref{fig4}, respectively. These results show that  point 2's vertical position bounces at -2.35m, as expected, and its vertical velocity jumps at each time of event.
\begin{figure} [H] 
\centering
	\begin{subfigure}{.49\textwidth}
	\centering
	\includegraphics[width=\textwidth]{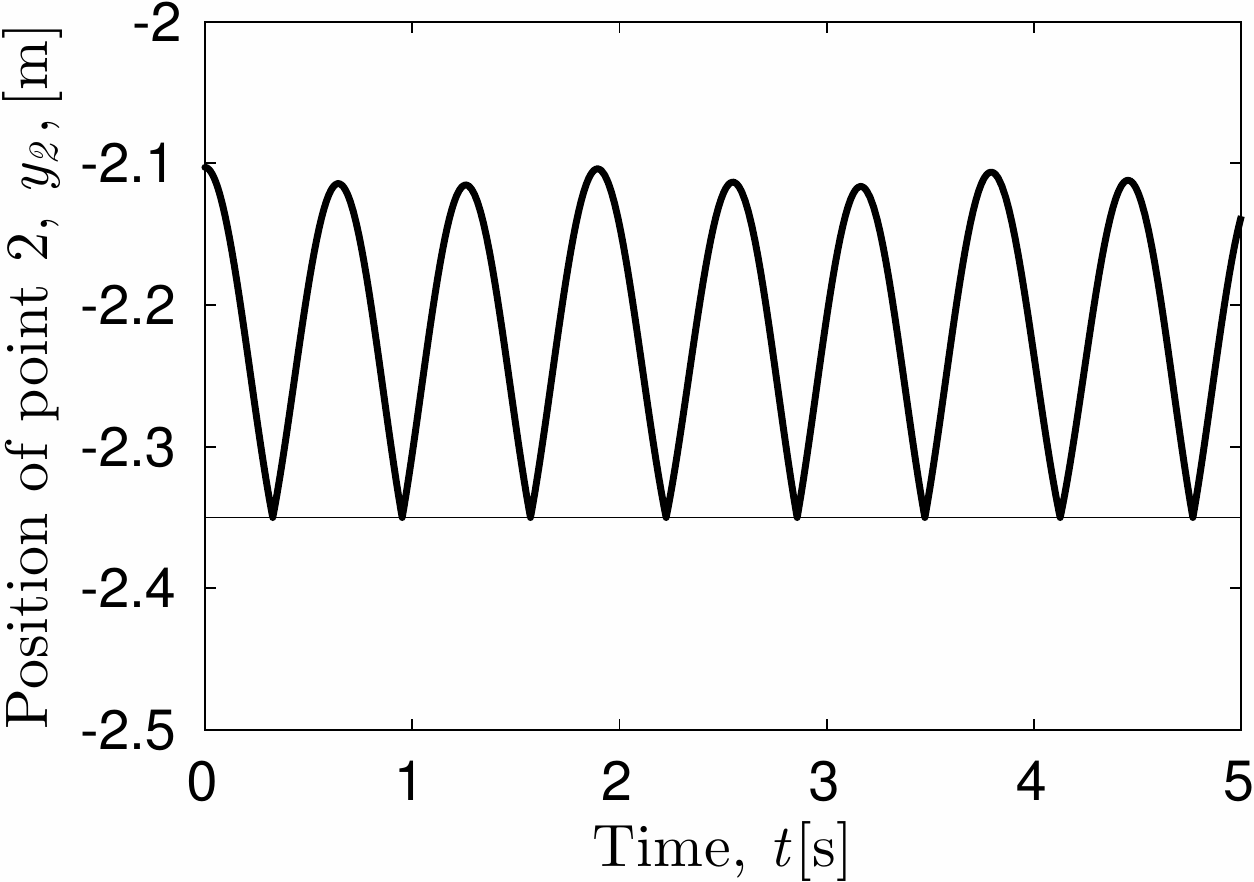}
	\captionsetup{margin=1cm}
	\caption{The vertical position of the bottom point $\pttwo$ of the five-bar mechanism}
	\label{fig3}
	\end{subfigure}
	\begin{subfigure}{.49\textwidth}
	\centering
	\includegraphics[width=\textwidth]{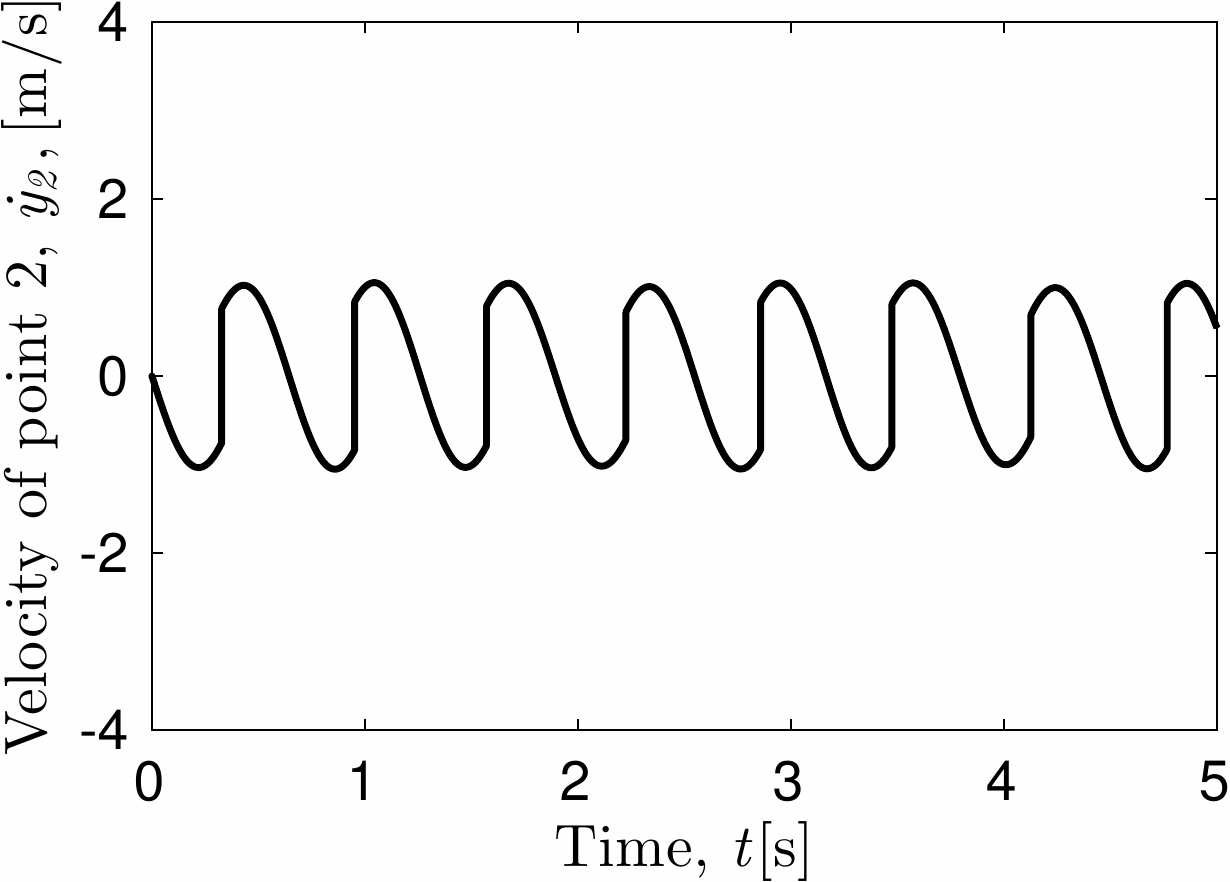}
	\captionsetup{margin=1cm}
	\caption{The vertical velocity of the bottom point $\dpttwo$ of the five-bar mechanism}
	\label{fig4}
	\end{subfigure}
\caption{Trajectories of the bottom point the five-bar mechanism}
\label{sfig2}
\end{figure}

The trajectories of the sensitivity of the position and velocity of point 2 of the five-bar mechanism along the vertical y axis are shown in Fig.~\ref{fig5} and Fig.~\ref{fig6}, respectively. The analytical sensitivity is represented by the continuous line, while the central finite difference sensitivity is represented by the dashed line. There is an excellent correlation between the numerical and the analytical sensitivities, with a difference between the two trajectories of less than $0.1\%$. Note that the numerical sensitivity of the velocity of point 2 along the vertical axis tends to be really large in magnitude, $1/\eps$ at each time of event. This is shown by the vertical dashed lines and it is due to the fact that the difference between the trajectories $v(\brho+\eps)$ and $v(\brho-\eps)$ increases considerably during $\Delta t$, as shown in Fig.~\ref{fig:im2}. This result shows that the novel analytical sensitivity method presented in this paper is considerably more robust than the numerical method, as it correctly calculates the sensitivity jumps and accurately determines the sensitivity trajectories after each invent without any delta-like jumps.
\begin{figure} [H] 
\centering
	\begin{subfigure}{.49\textwidth}
	\centering
	\includegraphics[width=\textwidth]{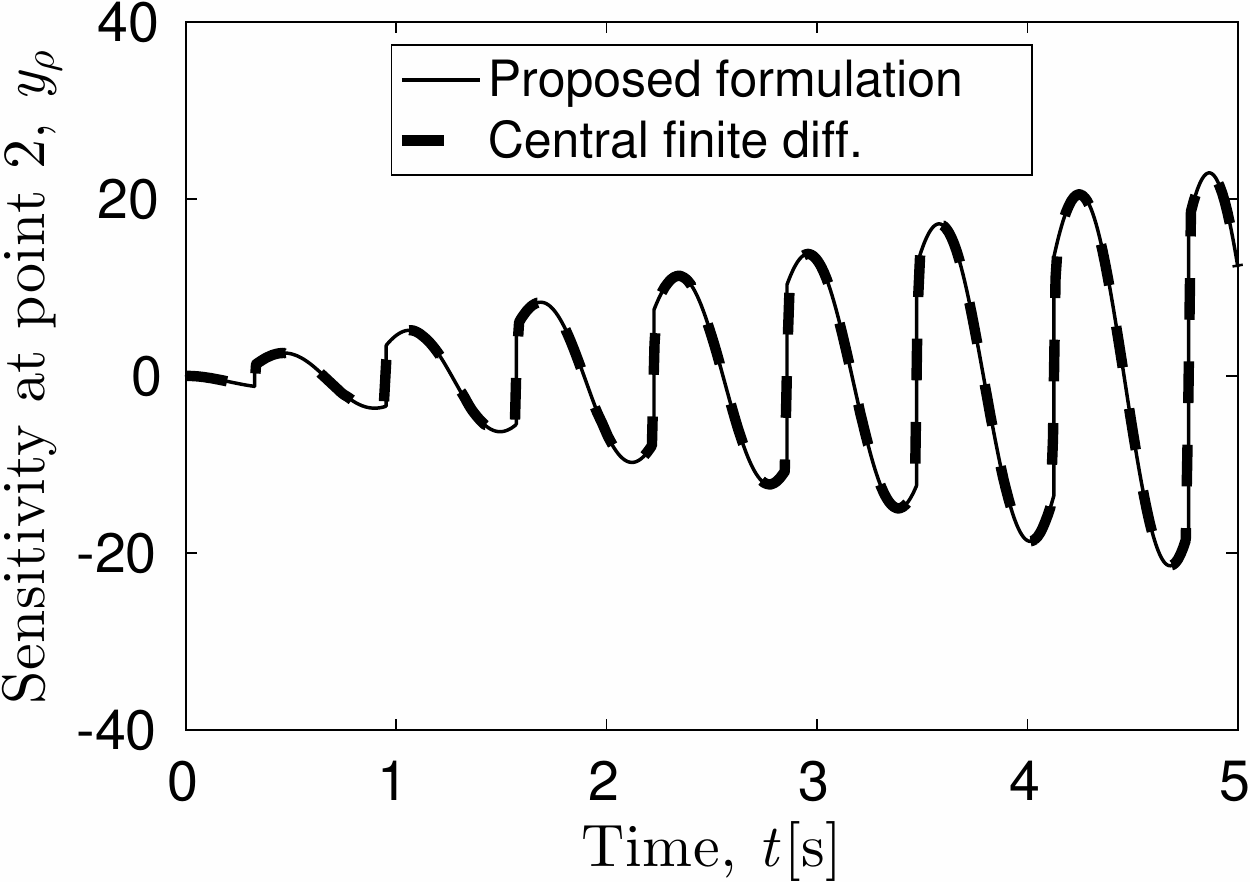}
	\captionsetup{margin=1.1cm}
	\caption{The sensitivity of the position of the bottom point $y_2$ of the five-bar mechanism}
	\label{fig5}
	\end{subfigure}
	\begin{subfigure}{.49\textwidth}
	\centering
	\includegraphics[width=\textwidth]{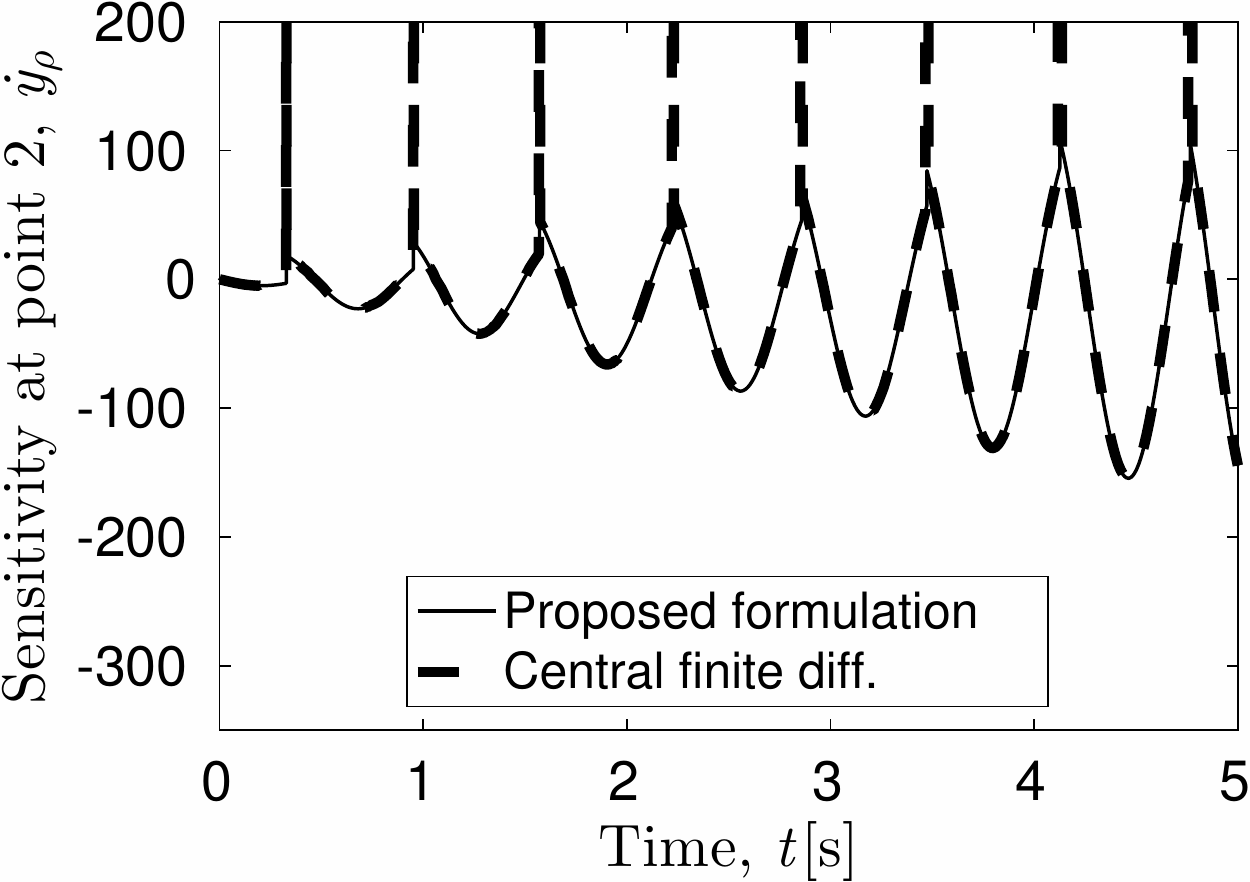}
	\captionsetup{margin=1.1cm}
	\caption{The sensitivity of the velocity of the bottom point $\dpttwo$ of the five-bar mechanism}
	\label{fig6}
	\end{subfigure}
\caption{Sensitivity analysis of the bottom point the five-bar mechanism}
\label{sfig3}
\end{figure}

The trajectories of the quadrature variables  $\zdytwo$  and $\zddytwo$ are shown in Fig.~\ref{fig7} and Fig.~\ref{fig8}, respectively. Note that $\zdytwo$  matches the trajectory of the position of point 2 along the vertical axis in Fig.~\ref{fig3}, while $\zddytwo$ does not completely match the trajectory of the velocity of point 2 in Fig.~\ref{fig4}. This is due to the fact that the velocity variable is affected by the impulse function at the time of event, while the quadrature variable is not. The trajectories after each invent differ by a constant since the quadrature variable evaluates the integral of the acceleration of point 2.
\begin{figure} [H] 
\centering
	\begin{subfigure}{.49\textwidth}
	\centering
	\includegraphics[width=\textwidth]{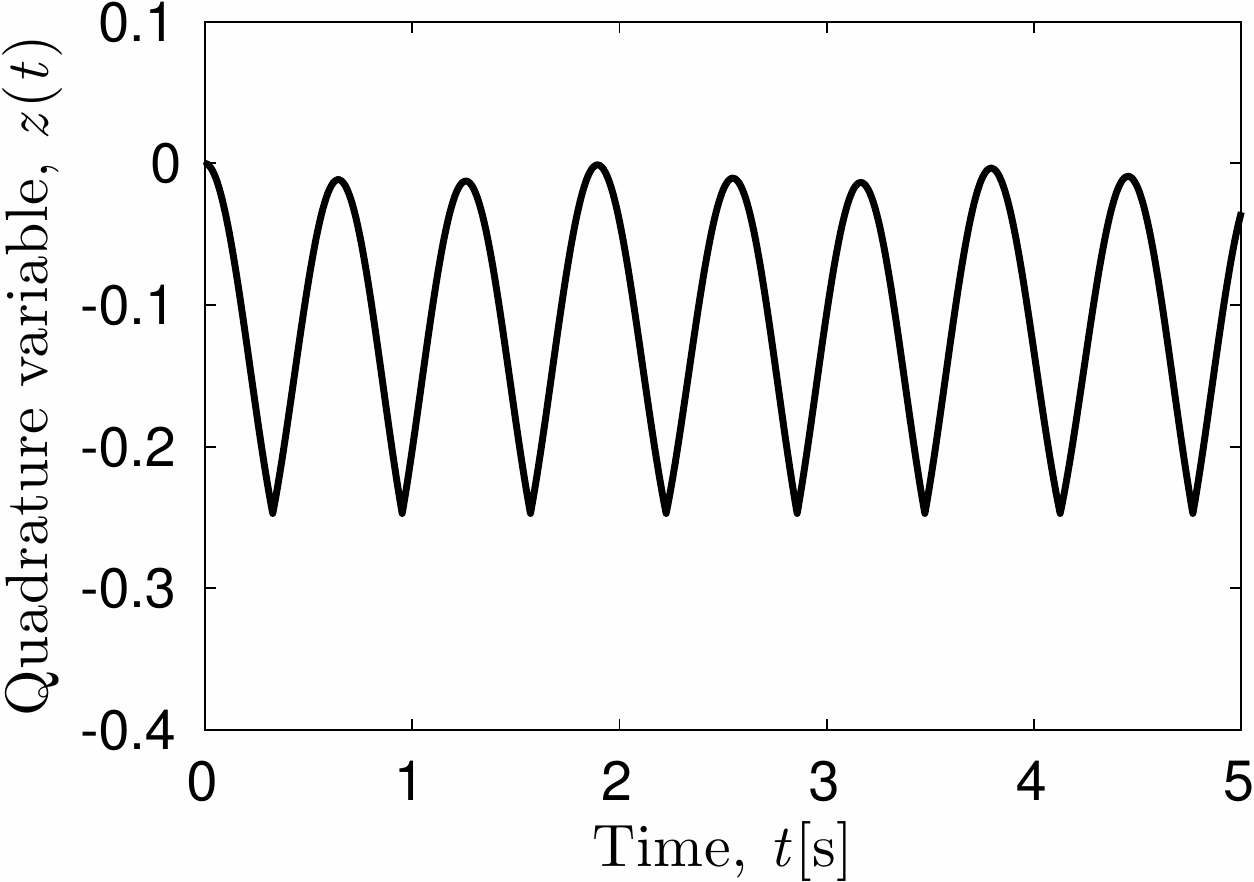}
	\captionsetup{margin=1cm}
	\caption{The quadrature variable $\zdytwo$ }
	\label{fig7}
	\end{subfigure}
	\begin{subfigure}{.49\textwidth}
	\centering
	\includegraphics[width=\textwidth]{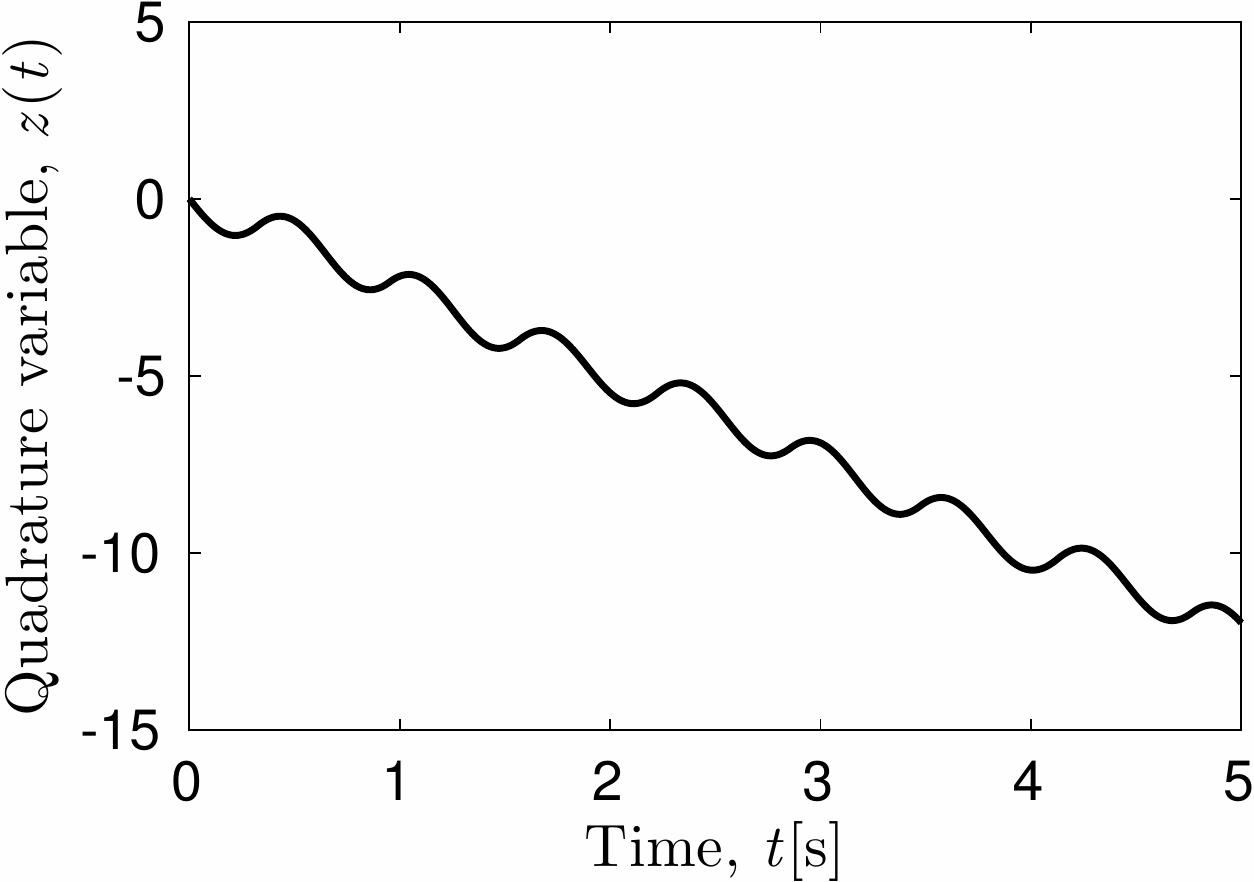}
	\captionsetup{margin=1cm}
	\caption{The quadrature variable $\zddytwo$ }
	\label{fig8}
	\end{subfigure}
\caption{The quadrature variables of the five-bar mechanism.}
\label{sfig4}
\end{figure}

The trajectories of the sensitivities of the quadrature variables  $\zdytwo$ and $\zddytwo$ are presented in  Fig.~\ref{fig9} and Fig.~\ref{fig10}, respectively. The sensitivities are with respect to the system parameters $\brho= \big[L_{01} ~~ L_{02} \big]$. The results of the analytical and numerical sensitivity analysis of the five-bar mechanism  highlight the quasi-perfect correlation between the numerical and analytical sensitivities with a difference between the two trajectories of less than $0.1\%$. Note that a similar observation with the one previously made is valid here: the sensitivity of $\zdytwo$ shown in Fig.~\ref{fig9} matches the trajectory of the sensitivity of the position illustrated in Fig.~\ref{fig5}, while the sensitivity of $\zddytwo$ shown in Fig.~\ref{fig10} does not completely match the trajectory of the sensitivity of the velocity from Fig.~\ref{fig6} because of the impulse function. 
\begin{figure} [H] 
\centering
	\begin{subfigure}{.49\textwidth}
	\centering
	\includegraphics[width=\textwidth]{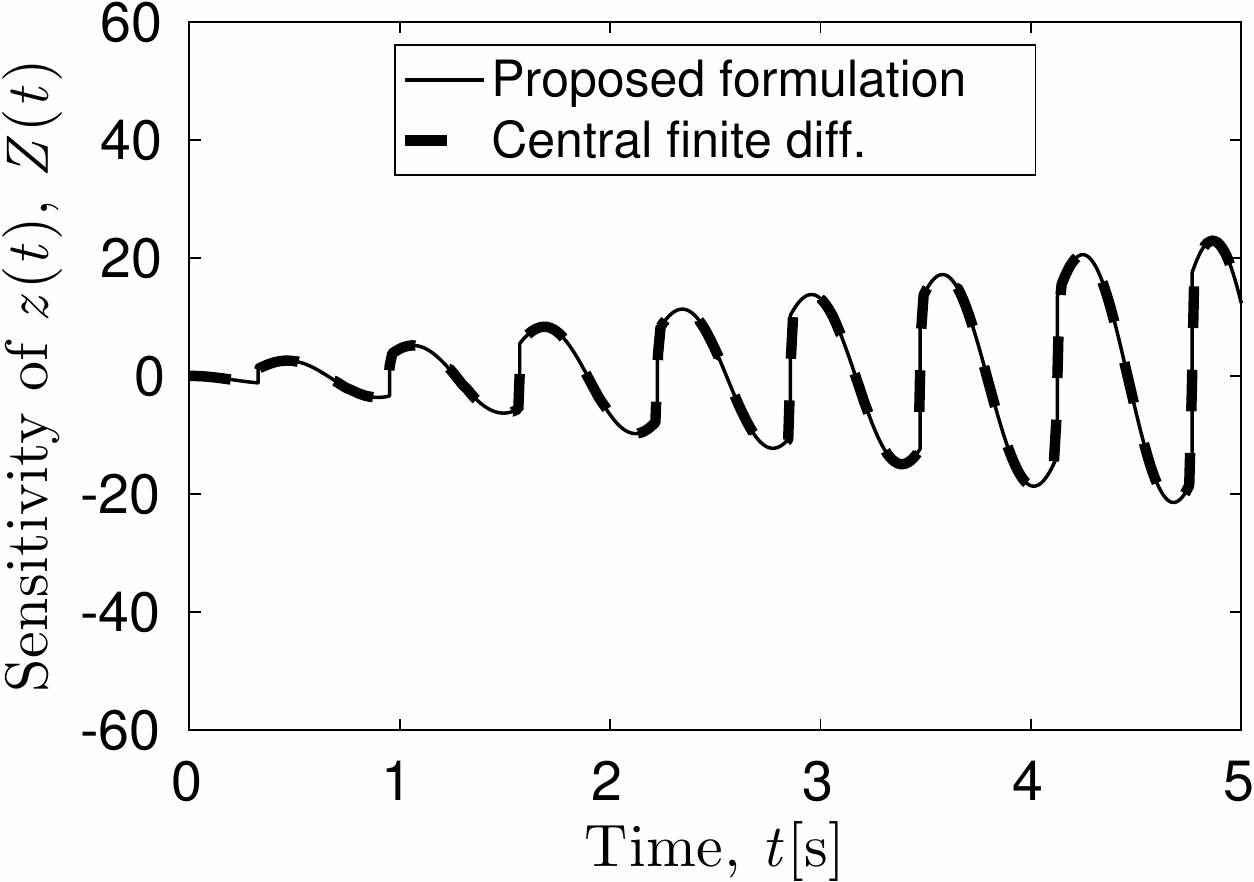}
	\captionsetup{margin=1cm}
	\caption{The sensitivity of $\zdytwo$ with $\bZ=\frac{d\bz}{d\brho}$.}
	\label{fig9}
	\end{subfigure}
	\begin{subfigure}{.49\textwidth}
	\centering
	\includegraphics[width=\textwidth]{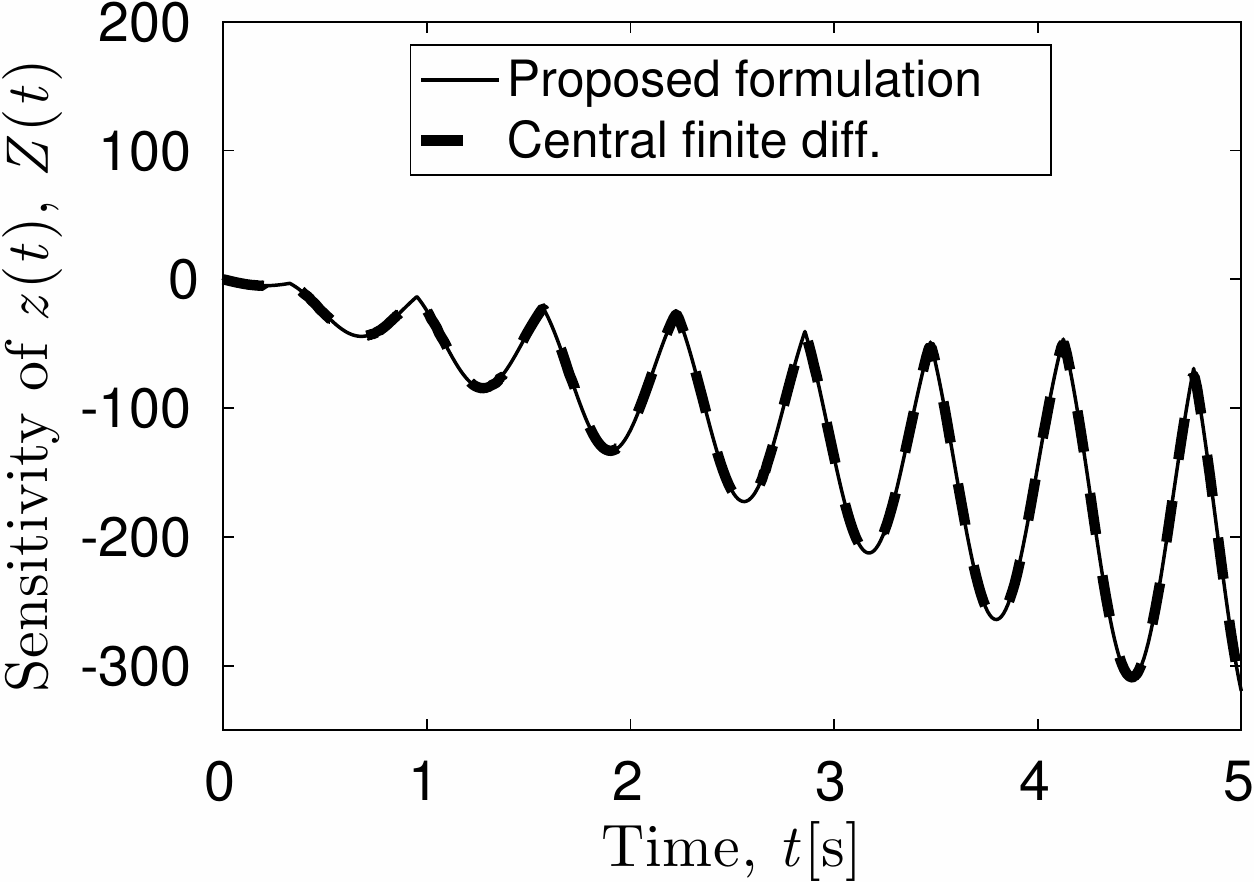}
	\captionsetup{margin=1cm}
	\caption{The sensitivity of $\zddytwo$ with $\bZ=\frac{d\bz}{d\brho}$}
	\label{fig10}
	\end{subfigure}
\caption{Sensitivity analysis of the quadrature variables of the five-bar mechanism}
\label{sfig5}
\end{figure}

\section{Conclusions}
\label{sec:conclusions}

An important ingredient in the analysis and design of multibody dynamical systems is the ability to compute sensitivities of model states with respect to system parameters. Direct and adjoint sensitivity analyses for mechanical systems with smooth trajectories have been studied in the literature \cite{Sandu_2013_sensitivity_ODE_multibody,Sandu_2014_sensitivity_ODE_multibody,zhu2014mbsvt,Zhu_2014_PhD,Sandu2015dynamic,Sandu_2017_vehicle-optimization}. This work focuses on the case where trajectories are non-smooth, and specifically, on hybrid systems where a finite number of events can instantly change the velocity state variables. 

We develop herein a complete mathematical framework needed to compute the sensitivities of model states, and of general cost functionals, with respect to model parameters for hybrid dynamical systems. The paper first considers multibody system is described by ordinary differential equations, as is the case when the dynamics is modeled using recursive coordinate systems. Next, the paper multibody systems described by differential algebraic equations, as is the case when the constraints imposed by the joints holding the system together are explicitly modeled. At the time moments when events happen, and there are points of discontinuity in the forward trajectory, the evolution of sensitivity variables is also discontinuous. Jump conditions for sensitivity variables are established for both the ordinary and the differential algebraic cases. These jump conditions specify the values of the sensitivities right after the event, given their value right before the event, and the characteristics of the impact.


A five-bar mechanism with non-smooth contacts is used as a case study. The analytical sensitivities obtained by the proposed methodology are validated against numerical sensitivities computed by real and complex finite differences.


Current  efforts  focus on  applying the  hybrid  system sensitivity analysis methodology to robotics, where sensitivities of the performance of a robotic system with respect to changes in the system configuration, under non-smooth impact conditions, are a topic of great interest. Future work will consider extending the current framework to perform adjoint sensitivity analysis of hybrid mechanical systems. 

\clearpage

\section*{Acknowledgments}
\label{sec:acknowledgments}

This work was supported in part by awards NSF DMS--1419003, NSF CCF--1613905, AFOSR DDDAS 15RT1037,
by the Computational Science Laboratory, and by the Advanced Vehicle Systems Laboratory at Virginia Tech.

%
\appendix
\section{Terminology used in Section \ref{sec:multibody-smooth} }
\label{sec:AppendixA}
In Eqs.~Eq.~\eqref{eq:TLM-ODE-penalty}  the terms $\overline{\Force}_{\bq}$, $\overline{\Force}_{\dbq}$, $\overline{\Force}_\brho$, $\overline{\Mass}_{\bq} \ddbq$, and $\overline{\Mass}_\brho \ddbq$ are given by the following expressions:
\begin{eqnarray}
\label{eq:barK}
&& \overline{\Force}_{\bq} = {\Force}_{\bq} - 
{\bPhi}_{\bq\bq}^{\rm T}\alpha\left(\dtdPhidq\dbq
\dPhi_{t}+2 \, \xi\, \omega\, \dPhi+\omega^2 \bPhi \right)-\nonumber \\
&&\dPhidq^{\rm T}\alpha\left(\left(\dtdPhidq\dbq\right)_{\bq}\!+\!
\left(\dPhi_{t}\right)_{\bq}\!+\!2 \xi \omega 
\left({\bPhi}_{\bq\bq}\dbq\!+\!{\bPhi}_{t{\bq}}\right)\!+\! 
\omega^2 \dPhidq\right), \\
\label{eq:barC}
&& \overline{\Force}_{\dbq} ={\Force}_{\dbq} - \! 
\dPhidq^{\rm T}\alpha\left({\bPhi}_{\bq\bq}\dbq\!+\!\dtdPhidq\!+\!
{\bPhi}_{tq}+2 \, \xi\, \omega\, \dPhidq\right), \\
\label{eq_dbarQdro}
&&\overline{\Force}_\brho = \Force_\brho- \bPhi_{\bq \brho}^{\rm T}\alpha\left(\dtdPhidq\dbq+
\dPhi_{t}+2 \, \xi\, \omega\, \dPhi+\omega^2 {\bPhi}\right)- \nonumber \\
&&\dPhidq^{\rm T}\alpha\left(\left(\dtdPhidq\dbq\right)_\brho+
\dPhi_{t\brho}+2 \, \xi\, \omega\, \dPhi_\brho+\omega^2 {\bPhi}_\brho\right), \\
\label{eq_dbarMdq_qs}
&&\overline{\Mass}_{\bq} \ddbq=
{\Mass}_{\bq}\ddbq+
{\bPhi}_{\bq\bq}^{\rm T} \left( \alpha\dPhidq\ddbq \right)+
\dPhidq^{\rm T} \alpha \left( {\bPhi}_{\bq\bq}\ddbq \right), \\
\label{eq_dbarMdro_qs}
&&\overline{\Mass}_\brho \ddbq =
{\Mass}_\brho\ddbq+
{\bPhi}_{\bq \brho}^{\rm T} \left( \alpha \dPhidq\ddbq \right) +
\dPhidq^{\rm T}\alpha \left( {\bPhi}_{\bq \brho}\ddbq \right).
\end{eqnarray}
In Eq.~\ref{eq:canonical-ode-penalty-sensitivity}  the terms $\blambda_\by^{*} $, $\blambda_\dby^{*} $, $\blambda_\dbv^{*}$, $\blambda_\bv^{*}$, $\blambda_{\bq}^{*}$ , and $\blambda_\brho^{*} $ are given by the following expressions:
\label{eq:MBD canonical ODE system}

\begin{eqnarray}
\label{eq:delta_lambda2}
&&\blambda_\by^{*} = \left[\begin{array}{c c}
\blambda_{\bq}^{*} & \blambda_\bv^{*}
\end{array}\right] \\
\label{eq:lambda_yp}
&&\blambda_\dby^{*} = \left[\begin{array}{c c}
\bzero & \blambda_{\dbv}^{*}
\end{array}\right] \\
\label{eq:lambda_vp}
&&\blambda_\dbv^{*} = \alpha \dPhidq \\
\label{eq:lambda_v}
&&\blambda_\bv^{*} = \alpha 
\left[{\bPhi}_\bqq \bv + \dtdPhidq + {\bPhi}_{t{\bq}}
+2 {\xi} {\omega} \dPhidq \right] \\
\label{eq:lambda_q}
&&\blambda_{\bq}^{*} = \alpha 
\left[ \bPhi_\bqq \dbv 
+\left( \dPhi_\bq\right)_\bq \bv+
\left( \bPhi_t \right)_\bq \right. \nonumber \\
&& \left. +2 \xi \omega \left(\bPhi_\bqq \bv + \bPhi_{t \bq}\right) 
+\omega^2 \dPhidq \right] \\
\label{eq:lambda_ro}
&&\blambda_\brho^{*} = \alpha 
\left[ \bPhi_{\bq \brho} \dbv+
\left(\dtdPhidq\right)_\brho \bv +\left(\dPhi_t \right)_\brho \right. \nonumber \\
&&\left. +2 \xi \omega \left(\bPhi_{\bq \brho} \bv +\bPhi_{t\brho}\right)
+\omega^2 \bPhi_\brho
\right]
\end{eqnarray}

\clearpage

\section*{References}
\bibliographystyle{elsarticle-num} 
\bibliography{biblio}


\end{document}